\documentclass[11pt]{article}

\usepackage{amsmath}
\usepackage{amsfonts}
\usepackage{graphicx}
\usepackage{setspace}
\usepackage{amsmath}
\usepackage{amssymb}
\usepackage{latexsym}
\usepackage{amsmath,amsfonts,amssymb,amsthm,braket,euscript,makeidx,color,mathrsfs,tikz}
\usepackage[title]{appendix}

\definecolor{YangRevisionAdded}{rgb}{0,0,1}

\usepackage[numbers,sort&compress]{natbib}

\def\sqr#1#2{{\vcenter{\vbox{\hrule height.#2pt
\hbox{\vrule width.#2pt height#1pt \kern#1pt \vrule width.#2pt}
\hrule height.#2pt}}}}
\def\3n{\negthinspace \negthinspace \negthinspace }
\def\2n{\negthinspace \negthinspace }
\def\1n{\negthinspace }

\def\dbU{\mathbb{U}}

\def\dbX{\mathbb{X}}     
\def\dbY{\mathbb{Y}}

\def\={\buildrel \triangle \over =}

\def\R{{\bf R}}

\def\min{\mathop{\rm min}}

\def\sup{\mathop{\rm sup}}

\usepackage[colorlinks,linkcolor=red,anchorcolor=gray,citecolor=red,urlcolor=blue]{hyperref}
\usepackage{caption}

\RequirePackage[capitalize,nameinlink]{cleveref}

\crefname{section}{section}{sections}
\crefname{subsection}{subsection}{subsections}
\Crefname{section}{Section}{Sections}
\Crefname{subsection}{Subsection}{Subsections}

\crefname{condition}{Condition}{Conditions}

\Crefname{figure}{Figure}{Figures}

\crefformat{equation}{\textup{#2(#1)#3}}
\crefrangeformat{equation}{\textup{#3(#1)#4--#5(#2)#6}}
\crefmultiformat{equation}{\textup{#2(#1)#3}}{ and \textup{#2(#1)#3}}
{, \textup{#2(#1)#3}}{ and \textup{#2(#1)#3}}
\crefrangemultiformat{equation}{\textup{#3(#1)#4--#5(#2)#6}}%
{ and \textup{#3(#1)#4--#5(#2)#6}}{, \textup{#3(#1)#4--#5(#2)#6}}{ and \textup{#3(#1)#4--#5(#2)#6}}

\Crefformat{equation}{#2Equation~\textup{(#1)}#3}
\Crefrangeformat{equation}{Equations~\textup{#3(#1)#4--#5(#2)#6}}
\Crefmultiformat{equation}{Equations~\textup{#2(#1)#3}}{ and \textup{#2(#1)#3}}
{, \textup{#2(#1)#3}}{ and \textup{#2(#1)#3}}
\Crefrangemultiformat{equation}{Equations~\textup{#3(#1)#4--#5(#2)#6}}%
{ and \textup{#3(#1)#4--#5(#2)#6}}{, \textup{#3(#1)#4--#5(#2)#6}}{ and \textup{#3(#1)#4--#5(#2)#6}}

\crefdefaultlabelformat{#2\textup{#1}#3}

\newtheorem {theorem}{Theorem}[section]
\newtheorem {lemma}{{\bf Lemma}}[section]
\newtheorem {corollary}{{\bf Corollary}}[section]
\newtheorem {proposition}{{\bf Proposition}}[section]
\theoremstyle{remark}
\newtheorem {remark}{{\bf Remark}}[section]

\theoremstyle{definition}

\theoremstyle{plain} \numberwithin {equation}{section}
\newtheorem{assumption}{Assumption}

\numberwithin{assumption}{section}

\allowdisplaybreaks
\everymath{\displaystyle}
\usepackage{graphicx}

\usepackage{ifpdf}
\ifpdf
\usepackage{epstopdf}
\fi

\usepackage{enumerate}

\newcommand{\dd}{\,\mathrm d}
\newcommand{\norm}[1]{\left\|#1\right\|}
\newcommand{\abs}[1]{\left|#1\right|}
\newcommand{\ip}[2]{\left\langle #1,#2\right\rangle}

\newcommand{\ThetaT}{\Theta_T}

\begin{document}

\title{\bf Exponential Turnpike in Optimal Boundary Control of Multi-Dimensional First-Order Hyperbolic Systems}
\author{
Haitian Yang\thanks{ Department of Mathematical Sciences, Tsinghua University, Beijing, 100084, China.  {\small\it E-mail:} {\small\tt yanght17@tsinghua.org.cn.}}
}

\date{\today}

\maketitle

\begin{abstract}
\setlength{\emergencystretch}{2em}
We establish an exponential turnpike property for optimal boundary control
of multi-dimensional first-order linear hyperbolic systems in which both the
control and observation operators are unbounded and the cost functions may be
nonquadratic. The main assumptions are internal and boundary dissipation
conditions. These conditions ensure that the associated input-to-state and
input-to-output maps are bounded, with operator norms independent of the time
horizon \(T\). This uniform boundedness is essential to the proof, which
combines weighted energy estimates for the forward and adjoint equations,
input--output duality, and strong convexity. Finally, we apply the result to
the linearized two-dimensional Saint--Venant equations with sluice-gate
control and illustrate the turnpike behavior numerically.
\end{abstract}


\bigskip


\section{Introduction}

\subsection{Problem formulation and main result}
First-order hyperbolic partial differential equations (PDEs) govern the evolution of a wide range of physical and engineering processes. Typical examples include the Saint--Venant equations for open-channel flows, the Aw--Rascle model for road traffic, compressible gas flow in pipeline networks, production and supply chains, and heat exchangers \cite{BastinCoron2016}. In recent years, optimal control problems governed by such systems also have attracted considerable attention
\cite{BongartiHintermueller2024,ClementCoronShang2013,SchmittUlbrich2021}. In this paper, we investigate optimal boundary control for this class of systems and establish the associated turnpike property.

Given a time horizon $T>0$ and a bounded polyhedral domain
$\Omega\subset\R^d$, whose boundary $\Gamma:=\partial\Omega$ consists of
finitely many planar faces, we are interested in the following
multi-dimensional (multi-D) first-order linear hyperbolic systems with
boundary control:
\begin{equation}\label{eq:evolution-system}
  \begin{cases}
    \displaystyle f_t(t,x)+\sum_{j=1}^d A_j f_{x_j}(t,x)=Qf(t,x),
    & (t,x)\in(0,T)\times\Omega,\\[1mm]
    \zeta_-(t,x)=K(x)\zeta_+(t,x)+u(t,x),
    & (t,x)\in(0,T)\times\Gamma,\\[1mm]
    f(0,x)=f_0(x),
    & x\in\Omega.
  \end{cases}
\end{equation}
Here $f=f(t,x)\in\R^n$ denotes the solution (or state) of
\eqref{eq:evolution-system}, and $A_j,Q\in\R^{n\times n}$ are constant
matrices, where $n\geq1$ is the number of equations. The system is
hyperbolic in the symmetrizable sense: there exists a constant symmetric
positive definite matrix $A_0\in\R^{n\times n}$, called a symmetrizer, such
that $A_0A_j$ is symmetric for every $j=1,\ldots,d$.

On $\Gamma$,
$\zeta_-=\zeta_-(t,x)\in\R^{n_-}$ and
$\zeta_+=\zeta_+(t,x)\in\R^{n_+}$ are called the incoming and outgoing
 variables, respectively. They are linear combinations of the
boundary trace of $f$. The outgoing variable $\zeta_+$ carries information leaving the domain and can therefore be measured by sensors placed on the boundary, while the incoming variable  $\zeta_-$ represents information entering the domain and can be modified by actuators through the boundary control  $u=u(t,x)$. Hence, the outgoing variable is observed, whereas the incoming variable is controlled. The precise definitions of $\zeta_\pm$ and the dimensions $n_\pm$
are given in Section~\ref{subsec:boundary-adjoint}. The prescribed gain matrix \(K=K(x)\in\R^{n_-\times n_+}\) is continuous on each planar face of \(\Gamma\) and couples the
outgoing variable \(\zeta_+\) to the incoming variable \(\zeta_-\).

We first define the state, control, and observation Hilbert spaces by
\begin{equation}\label{eq:abstract-spaces}
  \dbX=L^2(\Omega;\R^n),\qquad
  \dbU=L^2(\Gamma;\R^{n_-}),\qquad
  \dbY=L^2(\Gamma;\R^{n_+}),
\end{equation}
where the boundary spaces are understood piecewise as direct sums when
the characteristic multiplicities vary along \(\Gamma\).
We then introduce the differential operator $\mathcal A$ and its domain:
\begin{equation}\label{eq:homogeneous-realization-intro}
\begin{aligned}
  \mathcal A
  &:= -\sum_{j=1}^d A_j\partial_{x_j}+Q\,\cdot,\\
  D(\mathcal A)
  &:=\left\{
    f\in\dbX:
    \mathcal Af\in\dbX,
    \zeta_-=K\zeta_+\ \text{on }\Gamma
  \right\},
\end{aligned}
\end{equation}
where $Q\,\cdot$ denotes pointwise left multiplication by the matrix $Q$.
The system can then be written abstractly as
\begin{equation}\label{eq:abstract-control-system}
 \dot f(t)=\mathcal A f(t)+\mathcal B u(t),
 \qquad \zeta_+(t)=\mathcal C f(t).
\end{equation}
The operators \(\mathcal B\) and \(\mathcal C\) are precisely described in
Appendix~\ref{app:boundary-operators}; both are unbounded with respect
to the state space \(\dbX\).
According to
Theorem~\ref{thm:wellposedness}, the boundary control system \eqref{eq:abstract-control-system} is
well-posed.

Consider the dynamic optimal control problem
\begin{equation}\label{eq:dynamic-cost-overview}
\begin{aligned}
 \underset{u\in L^2(0,T;\dbU)}{\operatorname{inf}}\quad
 J_T(u)
 &:=
 \int_0^T\left[
   \int_\Omega \Phi\bigl(x,f(t,x)\bigr)\dd x
   +\int_\Gamma
      F\bigl(x,u(t,x),\zeta_+(t,x)\bigr)\dd\sigma
 \right]\dd t,\\
 \text{subject to}\quad &\eqref{eq:evolution-system},
\end{aligned}
\end{equation}
where $\dd\sigma$ denotes the standard surface measure on $\Gamma$. 
Under Assumption~\ref{ass:cost}, this problem admits a unique
optimal control \(u_T\in L^2(0,T;\dbU)\). Given \(u_T\), we denote the
corresponding state and observation of \eqref{eq:evolution-system} by
\[
 f_T\in L^2(0,T;\dbX),
 \qquad
 \zeta_{T,+}\in L^2(0,T;\dbY).
\]

We next consider the stationary system
\begin{equation}\label{eq:stationary-system}
\begin{cases}
  \displaystyle\sum_{j=1}^d A_j f_{x_j}=Qf,
  &x\in\Omega,\\[1mm]
  \zeta_-=K\zeta_++u,
  &x\in\Gamma.
\end{cases}
\end{equation}
This is a pure boundary-value problem whose well-posedness is not
automatic in several space dimensions. Appendix~\ref{app:static-counterexample}
provides a simple $2$-D Saint--Venant counterexample showing that
\eqref{eq:stationary-system} may admit infinitely many solutions and
therefore fail to be well-posed. We consequently impose the internal dissipation assumption, under which we establish the well-posedness of
\eqref{eq:stationary-system}.

Now we consider the corresponding static optimal control problem
\begin{equation}\label{eq:static-cost-overview}
\begin{aligned}
 \underset{u\in\dbU}{\operatorname{inf}}\quad
 J_s(u)
 &:=
 \int_\Omega \Phi(x,f(x))\dd x
 +\int_\Gamma F(x,u(x),\zeta_+(x))\dd\sigma,\\
 \text{subject to}\quad &\eqref{eq:stationary-system}.
\end{aligned}
\end{equation}
The static problem admits a unique optimal control, denoted by \(u_s\);
see Section~\ref{subsec:cost-assumptions}.
We denote the associated state and observation by \(f_s\) and \(\zeta_{s,+}\),
respectively.

The static problem is typically less expensive to solve. Furthermore, its optimal solution provides useful initialization information for computing the dynamic optimizer when the two are sufficiently close; see the shooting method in \cite{TrelatZuazua2015}. This motivates a quantitative comparison between the dynamic and static optimal solutions.

Define
\begin{equation*}
  d_T(t):=\min\{t,T-t\},
  \qquad
  I_t:=\bigl(d_T(t),T-d_T(t)\bigr).
\end{equation*}
Equivalently,
\[
  I_t=
  \begin{cases}
    (t,T-t), & 0\le t<\dfrac{T}{2},\\[1mm]
    \emptyset, & t=\dfrac{T}{2},\\[1mm]
    (T-t,t), & \dfrac{T}{2}<t\le T.
  \end{cases}
\]
These intervals are precisely those considered in
\cite{NguyenTrelat2026}. Our main result is the following exponential turnpike property: there exist
constants \(C,\mu>0\), independent of \(T\), such that
\begin{equation}\label{eq:main-turnpike-average-overview}
  \norm{f_T(t)-f_s}_{\dbX}
  +\norm{u_T-u_s}_{L^2(I_t;\dbU)}
  +\norm{\zeta_{T,+}-\zeta_{s,+}}_{L^2(I_t;\dbY)}
  \le
  C\left(e^{-\mu t}+e^{-\mu(T-t)}\right),
  \qquad 0\le t\le T.
\end{equation}
Throughout the paper, \(C\) and \(c\) denote generic positive constants
whose values may vary from one occurrence to another but are always
independent of \(T\). Note that \(\mathcal B\) and \(\mathcal C\) are unbounded control and
observation operators, respectively, pointwise-in-time estimates
for the control and observation cannot generally be expected; see
\cite[Remark~1.4]{NguyenTrelat2026}. Furthermore, for every
\(t\in[0,T/2]\), estimate
\eqref{eq:main-turnpike-average-overview} yields
\begin{equation*}
  \norm{u_T-u_s}_{L^2((t,T-t);\dbU)}
  +\norm{\zeta_{T,+}-\zeta_{s,+}}_{L^2((t,T-t);\dbY)}
  \le Ce^{-\mu t}.
\end{equation*}
Thus, the discrepancies between the dynamic and static
controls and observations are confined to temporal boundary layers
near \(t=0\) and \(t=T\).

\subsection{Literature review}\label{subsec:literature}
Turnpike properties describe the tendency of optimal solutions over long time horizons to remain close to an optimal steady state for most of the time. The term ``turnpike'' was coined in the economics literature \cite[Chapter~12]{DorfmanSamuelsonSolow1958} in the study of efficient programs of capital accumulation within the von Neumann growth model. A major revival of the subject in optimal control and PDEs was initiated by \cite{PorrettaZuazua2013}, where turnpike properties for linear systems in both finite and infinite dimensions are studied. Related questions for semilinear heat equations were subsequently investigated in \cite{PorrettaZuazua2016Remarks}. For finite-dimensional nonlinear optimal control problems, a systematic exponential turnpike theory was developed in \cite{TrelatZuazua2015}, based on the hyperbolic saddle structure of the Hamiltonian extremal system. Since then, the theory has been developed in several directions, including integral and measure turnpikes \cite{TrelatZhang2018}, dissipativity-based turnpike theory \cite{FaulwasserKordaJonesBonvin2017,GruneMueller2016}, periodic exponential turnpikes \cite{TrelatZhangZuazua2018}, stochastic turnpikes \cite{SunYong2024}, mean-field turnpikes \cite{HertyZhou2025}, and turnpike phenomena in deep learning \cite{EsteveYagueGeshkovskiPighinZuazua2022}. We refer to the surveys \cite{GeshkovskiZuazua2022,TrelatZuazuaSurvey} for broader accounts of these developments and further references.

In PDE optimal control, most exponential turnpike results assume bounded control and observation operators $\mathcal B$ and $\mathcal C$.  Recent abstract infinite-dimensional results relax this assumption by allowing admissible unbounded $\mathcal B$ with bounded $\mathcal C$, using multiplier methods \cite{GruneSchallerSchiela2020} or Yosida approximation \cite{NguyenTrelat2026}.
For specific  PDEs, their special structure allows one to go beyond the abstract framework. For the 1-D wave equation (a second-order hyperbolic system) with boundary control and observation, \cite{GugatTrelatZuazua2016} exploits explicit characteristic representations to establish an exponential turnpike property. For 1-D first-order hyperbolic systems, multiplier-based arguments in \cite{GugatHante2019,Gugat2019ConvexHyperbolic,GugatHerty2023} allow both boundary control and observation to be treated, but yield only integral rather than exponential turnpike properties.

\subsection{Organization of the paper}\label{subsec:organization}

The present work connects abstract turnpike theory with unbounded operators and optimal boundary control of first-order hyperbolic PDEs. Assumptions~\ref{ass:stabilizing-symmetrizer} and
\ref{ass:boundary-dissipativity} ensure that the generator
\(\mathcal A\) is exponentially stable, so \((\mathcal A,\mathcal B)\)
and \((\mathcal A,\mathcal C)\) are exponentially stabilizable and
detectable, respectively. The special boundary input--output structure of
first-order hyperbolic systems yields uniform energy estimates for
the forward and adjoint systems; see
Lemmas~\ref{lem:forward-energy-estimate} and
\ref{lem:backward-energy-estimate}. Using these estimates, we develop
a time-weighted energy argument to establish an exponential turnpike
property for this class of systems with unbounded boundary control
and observation operators.
Compared with the 1-D results in \cite{GugatHante2019,Gugat2019ConvexHyperbolic,GugatHerty2023}, it strengthens the known integral turnpike behavior to an exponential one and extends the analysis to multiple space dimensions.  Since our proof avoids Riccati equations, it also applies beyond the linear-quadratic setting to nonquadratic costs satisfying Assumption~\ref{ass:cost}.

The remainder of the paper is organized as follows.
Section~\ref{sec:preliminaries} develops the dissipative forward--adjoint
framework, and derives the first-order optimality conditions for the dynamic
and static problems.
Section~\ref{exp:sec:exponential-turnpike} proves the weighted estimate
underlying Theorem~\ref{thm:main-exponential}, and establishes the exponential
turnpike property. Section~\ref{sec:saint-venant} applies these results to the
linearized 2-D Saint--Venant equations with sluice-gate control and presents
a numerical illustration. Section~\ref{sec:conclusion} concludes the paper.
Appendix~\ref{app:stationary-wellposedness} first gives a Saint--Venant
counterexample showing that stationary well-posedness is not automatic and
then establishes the well-posedness of the stationary state and adjoint
boundary value problems under the internal and boundary dissipation
conditions. Appendix~\ref{app:boundary-operators} provides a description of the unbounded boundary control and observation operators.

\subsection{Acknowledgments}

This paper was written during the author's visit to DCN-AvH at Friedrich-Alexander-Universit{\"a}t
Erlangen-N{\"u}rnberg. The author thanks Dr.~Yizhou Zhou and Professor Martin Gugat for
introducing him to this problem. He is grateful to Professor Enrique
Zuazua for his warm hospitality, for helpful discussions on the turnpike property and for
highlighting its connection with stabilization. He also thanks
Dr.~Zhengping Ji for pointing out relevant references, and Professors
Emmanuel Tr{\'e}lat and Hoai-Minh Nguyen for advice on avoiding Riccati
equations.

\section{Preliminaries}
\label{sec:preliminaries}

To prepare for the proof of the exponential turnpike property, we
introduce the internal and boundary dissipation assumptions, establish
the corresponding energy estimates, and derive the first-order
optimality conditions.

\subsection{Internal dissipation}
\label{subsec:stabilizing-symmetrizer}
We begin with the following structural assumption, which provides the
strict bulk dissipation required by the energy estimates in Section \ref{subsec:energy-estimates}.
\begin{assumption}[Internal dissipation]
\label{ass:stabilizing-symmetrizer} 
There exists a real symmetric, uniformly positive definite matrix-valued function
\(E=E(x)\in C^1(\overline{\Omega};\R^{n\times n})\) such that
\(EA_j\) is symmetric for 
\(j=1,\ldots,d\). Moreover, for some \(C>0\),
\begin{equation}\label{eq:E-internal-dissipation}
 Q^\top E+EQ+\sum_{j=1}^d\partial_{x_j}(EA_j)
 \le -CI
 \qquad\text{in }\Omega.
\end{equation}
Throughout the paper, the superscript \({}^\top\) denotes matrix or vector
transposition. We call \(E=E(x)\) a stabilizing symmetrizer.
\end{assumption}
 To find stabilizing symmetrizers $E=E(x)$ plays a
central role in the stabilization theory of multi-D first-order hyperbolic
systems; see
\cite{YangYong2024,YangYong2025,HertyThein2024}.

The turnpike analysis couples the forward dynamics with the backward
adjoint dynamics. The next proposition shows that the latter inherits
the same internal dissipative structure through the inverse symmetrizer
\(E^{-1}\). This property will be used to derive the backward energy
estimates in Section~\ref{subsec:energy-estimates}.

\begin{proposition}
\label{prop:inverse-adjoint-symmetrizer}
Let \(E\) be the stabilizing symmetrizer from
Assumption~\ref{ass:stabilizing-symmetrizer}. Then its inverse
\(E^{-1}=E^{-1}(x)\) symmetrizes the transposed transport matrices:
\(E^{-1}A_j^\top\) is symmetric for every
\(x\in\overline{\Omega}\) and \(j=1,\ldots,d\). Moreover, for
some \(C>0\),
\begin{equation}\label{eq:inverse-internal-dissipation}
 QE^{-1}+E^{-1}Q^\top
 -\sum_{j=1}^d\partial_{x_j}\!\left(E^{-1}A_j^\top\right)
 \le -CI
 \qquad\text{in }\Omega.
\end{equation}
\end{proposition}

\begin{proof}
It is straightforward to see that \(E^{-1}A_j^\top\) is symmetric for every \(j\). Moreover, since the matrices \(A_j\) are constant, we have
\[
 -\partial_{x_j}(E^{-1}A_j^\top)
 =E^{-1}\partial_{x_j}(EA_j)E^{-1}.
\]
Therefore,
$$
 QE^{-1}+E^{-1}Q^\top
 -\sum_{j=1}^d\partial_{x_j}\!\left(E^{-1}A_j^\top\right)\\
 =
 E^{-1}\left[
 Q^\top E+EQ+\sum_{j=1}^d\partial_{x_j}(EA_j)
 \right]E^{-1}.
$$
By \eqref{eq:E-internal-dissipation}, we have \eqref{eq:inverse-internal-dissipation}
after changing
\(C\) if necessary.
\end{proof}

\subsection{Boundary conditions}
\label{subsec:boundary-adjoint}

We now explain the boundary conditions.
For almost every \(x\in\Gamma\), let
$\nu(x)=(\nu_1(x),\ldots,\nu_d(x))$ be the unit outward normal vector at
$x$. Note that $\nu(x)$ is a constant vector on each planar face of \(\Gamma\). Let
$$
A_\nu=A_\nu(x):=\sum_{j=1}^d\nu_j(x)A_j.
$$
Since \(E\) in Assumption~\ref{ass:stabilizing-symmetrizer} symmetrizes every \(A_j\), \(EA_\nu\) is symmetric, and hence \(E^{1/2}A_\nu E^{-1/2}\) is also symmetric. Thus \(A_\nu\) is diagonalizable with real
eigenvalues. Consequently, there exists an invertible matrix $\Pi = \Pi(x) \in \R^{n\times n}$, constant on each planar face of \(\Gamma\), such that
\begin{equation} \label{boundarymatrix}
\Lambda=\Lambda(x) = \Pi^{-1}(x) A_\nu(x)\Pi(x)
= \begin{pmatrix}
\Lambda_+(x) & 0 & 0 \\
0 & 0 & 0 \\
0 & 0 & \Lambda_-(x)
\end{pmatrix},\vspace{-2mm}
\end{equation}
where $\Lambda_+=\Lambda_+(x) \in \R^{n_{+}\times n_{+}}$ and $\Lambda_-=\Lambda_-(x) \in \R^{n_{-} \times n_{-}}$ are diagonal matrices containing the positive and negative eigenvalues of
\(A_\nu(x)\), respectively. The corresponding spectral
multiplicities $n_{+}=n_{+}(x)$ and $n_{-}=n_{-}(x)$ are constant on
each planar face, since the outward normal vector $\nu$ and hence
$A_\nu$ are constant there.
Since $\Gamma$ consists of finitely many
faces, the absolute values of the diagonal entries of $\Lambda_\pm$ are
uniformly bounded above and below by positive constants. Throughout the
paper, $n_\pm$ and all boundary quantities involving these dimensions
are understood facewise along $\Gamma$. Accordingly, the boundary
spaces \(\dbU\) and \(\dbY\) in \eqref{eq:abstract-spaces} are equipped
with the corresponding piecewise \(L^2\) inner products and are Hilbert
spaces.

We define the characteristic variables on \(\Gamma\) by
\begin{equation*}
  \zeta=
  \begin{pmatrix}
    \zeta_+\\
    \zeta_0\\
    \zeta_-
  \end{pmatrix}
  :=\Pi^{-1}f,
\end{equation*}
where \(\zeta_+\in\R^{n_+}\) and \(\zeta_-\in\R^{n_-}\) are the
outgoing and incoming variables, respectively. Following \cite{BenzoniGavageSerre2007}, the proper boundary conditions prescribe the incoming
variable in terms of the outgoing variable and the boundary
input, as expressed by \(\zeta_-=K\zeta_++u\) in
\eqref{eq:evolution-system}.

By the symmetry of \(EA_\nu\), \eqref{boundarymatrix} and Lemma~2.1 of \cite{YangYong2024} imply that
\begin{equation*}
 \Pi^\top E\Pi
 =
 \begin{pmatrix}
  X_+&0&0\\
  0&X_0&0\\
 0&0&X_-
 \end{pmatrix},
\end{equation*}
where  \(X_+ \in \R^{n_+\times n_+}, X_-\in \R^{n_-\times n_-}\) and $X_0$ are symmetric positive definite. Moreover, $X_+\Lambda_+$ and $-X_-\Lambda_-$ are both symmetric and positive definite. Define
\begin{equation*}
 M:=X_+\Lambda_+,
 \qquad
 N:=-X_-\Lambda_-.
\end{equation*}
Given the gain matrix $K=K(x)\in \R^{n_-\times n_+}$ in \eqref{eq:evolution-system}, we make the following assumption.
\begin{assumption}[Dissipative boundary
condition]
\label{ass:boundary-dissipativity}
There exists \(\delta>0\) such that
\begin{equation}\label{eq:strict-dissipativity-K}
 M-K^\top NK\ge\delta I_{n_+}
 \qquad\text{on }\Gamma,
\end{equation}
where \(I_{n_+}\) denotes the \(n_+\times n_+\) identity matrix.
Note that \eqref{eq:strict-dissipativity-K} holds whenever the norm of \(K\) is
sufficiently small.
\end{assumption}
For a boundary trace \(f\) with \(\zeta=\Pi^{-1}f\), we have
\begin{equation*}
\begin{aligned}
 f^\top EA_\nu f
 =(\Pi^{-1}f)^\top(\Pi^\top E\Pi)
   (\Pi^{-1}A_\nu\Pi)(\Pi^{-1}f)
 =\zeta_+^\top M\zeta_+-\zeta_-^\top N\zeta_-.
\end{aligned}
\end{equation*}
Thus, \(f^\top EA_\nu f\) is the outward energy-flux density associated
with \(E\): the first term represents the energy carried out by the
outgoing variable, whereas the second represents the energy carried
into the system by the incoming variable. If \(\zeta_-=K\zeta_++u\),  after decreasing \(\delta\), 
Assumption~\ref{ass:boundary-dissipativity} and Young's inequality give
\begin{equation}\label{eq:forward-boundary-algebra}
 -f^\top EA_\nu f
 \le -\delta\abs{\zeta_+}^2+C\abs{u}^2
 \qquad\text{on }\Gamma.
\end{equation}
When \(u=0\), \eqref{eq:forward-boundary-algebra} implies that the energy carried back into the system is strictly smaller than the energy
carried out by \(\zeta_+\). The boundary therefore produces a strict
net loss of energy, which explains the term \emph{dissipative boundary
condition}.

We now identify the adjoint operator \(\mathcal A^*\). One has
\begin{equation}\label{eq:standard-L2-adjoint}
 \mathcal A^*g
 =\sum_{j=1}^dA_j^\top g_{x_j}+Q^\top g.
\end{equation}
\begin{equation*}
\begin{aligned}
 D(\mathcal A^*)=\biggl\{g\in\dbX:
 \sum_{j=1}^dA_j^\top g_{x_j}+Q^\top g\in\dbX,
 \Lambda_+\xi_++K^\top\Lambda_-\xi_-=0
 \text{ on }\Gamma\biggr\},
\end{aligned}
\end{equation*}
where the adjoint boundary variables are
\begin{equation*}
 \xi=
 \begin{pmatrix}\xi_+\\ \xi_0\\ \xi_-\end{pmatrix}
 :=\Pi^\top g,
\end{equation*}
with $\xi_+ \in \R^{n_+}$ and $\xi_- \in \R^{n_-}.$ Since the transport part of \(\mathcal A^*\) in
\eqref{eq:standard-L2-adjoint} has the opposite sign to that of
\(\mathcal A\) in \eqref{eq:homogeneous-realization-intro}, the roles of
the boundary variables are reversed: for \(\mathcal A^*\), \(\xi_+\) is
the incoming variable, whereas \(\xi_-\) is the outgoing variable. The
boundary pairing is
\begin{equation}\label{eq:standard-boundary-pairing}
 f^\top A_\nu^\top g
 =\zeta_+^\top\Lambda_+\xi_+
 +\zeta_-^\top\Lambda_-\xi_-.
\end{equation}
Thus, for \(f\in D(\mathcal A)\) and \(g\in D(\mathcal A^*)\), the
divergence theorem and the boundary conditions give
\begin{equation*}
   \braket{\mathcal A f,g}_{\dbX}
 -\braket{f,\mathcal A^*g}_{\dbX}
 =-\int_\Gamma f^\top A_\nu^\top g\dd\sigma=0.
\end{equation*}
To analyze the boundary behavior of \(\mathcal A^*\), we first show
that the dual boundary condition inherits the dissipativity of the
original one.
\begin{lemma}
\label{lem:dual-boundary-dissipativity}
Under Assumption~\ref{ass:boundary-dissipativity}, after decreasing
\(\delta\) if necessary, we have
\begin{equation}\label{eq:inverse-boundary-dissipation}
 N^{-1}-KM^{-1}K^\top\ge\delta I_{n_-},
\end{equation}
where \(I_{n_-}\) denotes the \(n_-\times n_-\) identity matrix.
Consequently, given \(z\in\R^{n_+}\), if \(\Lambda_+\xi_++K^\top\Lambda_-\xi_-=z\), then, after decreasing
\(\delta\) if necessary, there exists \(C>0\) such that
\begin{equation} \label{eq:dualdissipation}
  g^\top E^{-1}A_\nu^\top g\le C\abs{z}^2
  -\delta\abs{\xi_-}^2.
\end{equation}
\end{lemma}

\begin{proof}
Set $
 R:=N^{1/2}KM^{-1/2}$ and we have $
 M-K^\top NK
 =M^{1/2}(I-R^\top R)M^{1/2}.$ 
By \eqref{eq:strict-dissipativity-K}, $I-R^\top R$ is strictly positive definite. Note that the nonzero eigenvalues of
\(R^\top R\) and \(RR^\top\) coincide; hence $I-RR^\top$ is also strictly positive definite. Finally, after decreasing  $\delta$ if necessary, we have
$$
 N^{-1}-KM^{-1}K^\top
 =N^{-1/2}(I-RR^\top)N^{-1/2} \geq \delta I_{n_-}.
$$

For the last assertion, we have
\begin{align*}
g^\top E^{-1}A_\nu^\top g
&=(\Pi^\top g)^\top
  (\Pi^{-1}E^{-1}\Pi^{-\top})
  (\Pi^\top A_\nu^\top\Pi^{-\top})
  (\Pi^\top g)\\
&=(\Lambda_+\xi_+)^\top M^{-1}(\Lambda_+\xi_+)
  -(\Lambda_-\xi_-)^\top
  N^{-1}(\Lambda_-\xi_-)\\
&=\bigl(z-K^\top\Lambda_-\xi_-\bigr)^\top
  M^{-1}
  \bigl(z-K^\top\Lambda_-\xi_-\bigr)
  -(\Lambda_-\xi_-)^\top
  N^{-1}(\Lambda_-\xi_-)\\
&=z^\top M^{-1}z
  -2z^\top M^{-1}K^\top\Lambda_-\xi_-
  +(\Lambda_-\xi_-)^\top
  (KM^{-1}K^\top-N^{-1})
  (\Lambda_-\xi_-)\\
&\le C\abs{z}^2
  -\delta\abs{\xi_-}^2.
\end{align*}
The last inequality follows from
\eqref{eq:inverse-boundary-dissipation} and Young's inequality, after
decreasing \(\delta\) if necessary.
\end{proof}

Finally, we record the following well-posedness result without proof, which follows
by adapting the arguments in \cite{HuangTemam2014} to the present
polyhedral domain.
\begin{theorem}\label{thm:wellposedness}
Under Assumptions~\ref{ass:stabilizing-symmetrizer} and
\ref{ass:boundary-dissipativity}, \(\mathcal A\) generates a strongly
continuous semigroup \(\bigl(e^{t\mathcal A}\bigr)_{t\ge0}\) on
\(\dbX\). For every \(T>0\), \(f_0\in\dbX\), and
\(u\in L^2(0,T;\dbU)\), problem \eqref{eq:evolution-system} admits a
unique mild solution \(f\in C([0,T];\dbX)\), and the corresponding
observation satisfies \(\zeta_+\in L^2(0,T;\dbY)\). Moreover, there exists a
constant \(C_T>0\), possibly depending on \(T\), such that
\begin{equation}\label{eq:dynamic-wellposedness-estimate}
 \sup_{0\leq t\leq T}\norm{f(t)}_{\dbX}^2
 +\norm{\zeta_+}_{L^2(0,T;\dbY)}^2
 \leq C_T\left(
 \norm{f_0}_{\dbX}^2
 +\norm{u}_{L^2(0,T;\dbU)}^2\right).
\end{equation}
\end{theorem}

\begin{remark}
Lemma~\ref{lem:forward-energy-estimate} below shows that the constant
\(C_T\) in \eqref{eq:dynamic-wellposedness-estimate} can be chosen
independently of \(T\). In the terminology of control theory
\cite{TucsnakWeiss2009}, \(\mathcal C\) is an infinite-time
\(L^2\)-admissible observation operator, since
\begin{equation*}
 \int_0^T\norm{\mathcal C e^{t\mathcal A}f_0}_{\dbY}^2\dd t=\norm{\zeta_+}_{L^2(0,T;\dbY)}^2
 \leq C\norm{f_0}_{\dbX}^2,
 \qquad f_0\in D(\mathcal A).
\end{equation*}
 Since Theorem~\ref{thm:wellposedness} also applies to the adjoint system and
\(\xi_-\) is the outgoing variable for this system, the corresponding
 estimate gives
\begin{equation*}
 \int_0^T\norm{\mathcal B^*e^{t\mathcal A^*}g_0}_{\dbU}^2\dd t=\norm{-\Lambda_-\xi_-}^2_{L^2(0,T;\dbU)}
 \leq C\norm{g_0}_{\dbX}^2,
 \qquad g_0\in D(\mathcal A^*),
\end{equation*}
which implies that \(\mathcal B\) is an infinite-time \(L^2\)-admissible control
operator.
\end{remark}

\subsection{Forward and adjoint energy estimates}
\label{subsec:energy-estimates}

Using the internal and boundary dissipation established above, we now
derive energy estimates for the forward and adjoint systems. To this
end, for \(f,g\in\dbX\), we define the weighted energies
\begin{equation*}
 \mathcal E(f):=\int_\Omega f^\top Ef\dd x,
 \qquad
 \mathcal E^*(g):=\int_\Omega g^\top E^{-1}g\dd x.
\end{equation*}
Both energies are equivalent to the squared \(\dbX\)-norm. For system \eqref{eq:evolution-system}, we have the following estimate.
\begin{lemma}\label{lem:forward-energy-estimate}
Given \(f_0\in\dbX\) and \(u\in L^2(0,T;\dbU)\), let \(f \in C([0,T];\dbX)\) solve \eqref{eq:evolution-system}. Then the outgoing variable \(\zeta_+\) satisfies
\begin{equation}\label{eq:forward-uniform-estimate}
\begin{aligned}
 \norm{f(T)}_{\dbX}^2
 +\norm{f}_{L^2(0,T;\dbX)}^2
 +\norm{\zeta_+}_{L^2(0,T;\dbY)}^2
 \le C\left(
 \norm{f_0}_{\dbX}^2
 +\norm{u}_{L^2(0,T;\dbU)}^2\right),
\end{aligned}
\end{equation}
where \(C\) is independent of \(T\).
\end{lemma}

\begin{remark}
Note that $f$ is only the mild solution. The calculation below could be first carried out for smooth compatible
data \((f_0,u)\), for which the corresponding solution has enough
regularity to justify integration by parts. The general case follows by a standard density argument; see
\cite[Section~2.1.3, p.~64]{BastinCoron2016}.
\end{remark}

\begin{proof}
By symmetry of \(EA_j\) and integration by
parts, we have
\begin{equation*}
\begin{aligned}
 \frac{\dd}{\dd t}\mathcal E(f(t))
 =\int_\Omega f^\top
 \left[
 Q^\top E+EQ+\sum_{j=1}^d(EA_j)_{x_j}
 \right]f\dd x-\int_\Gamma
 \left(\zeta_+^\top M\zeta_+
 -\zeta_-^\top N\zeta_-\right)\dd\sigma.
\end{aligned}
\end{equation*}
Using \eqref{eq:E-internal-dissipation},
\(\zeta_-=K\zeta_++u\), and
\eqref{eq:forward-boundary-algebra}, we have
\begin{equation}\label{eq:forward-complete-energy}
 \frac{\dd}{\dd t}\mathcal E(f(t))
 +c\norm{f(t)}_{\dbX}^2
 +c\norm{\zeta_+(t)}_{\dbY}^2
 \le C\norm{u(t)}_{\dbU}^2.
\end{equation}
Integrating \eqref{eq:forward-complete-energy} over \((0,T)\) and using the norm equivalence of $\mathcal{E}(f(t))$
yields
\eqref{eq:forward-uniform-estimate}.
\end{proof}

For the turnpike analysis, we also need an estimate for the
inhomogeneous backward adjoint system
\begin{equation}\label{eq:dynamic-dual-general-estimate}
\begin{cases}
 g_t+\displaystyle\sum_{j=1}^dA_j^\top g_{x_j}
 =-Q^\top g-q,
 &(t,x)\in(0,T)\times\Omega,\\[1mm]
 \Lambda_+\xi_++K^\top\Lambda_-\xi_-=z,
 &(t,x)\in(0,T)\times\Gamma,\\[1mm]
 g(T)=g_{T},
\end{cases}
\qquad \xi=\Pi^\top g=\begin{pmatrix}
  \xi_+ \\
  \xi_0\\
  \xi_-
\end{pmatrix},
\end{equation}
where \(q\in L^2(0,T;\dbX)\), \(z\in L^2(0,T;\dbY)\), and
\(g_{T}\in\dbX\). After reversing time, 
Theorem~\ref{thm:wellposedness} gives its
well-posedness.
\begin{lemma}
\label{lem:backward-energy-estimate}
Given \(q\in L^2(0,T;\dbX)\), \(z\in L^2(0,T;\dbY)\) and
\(g_{T}\in\dbX\), let \((g,\xi_-)\) solve \eqref{eq:dynamic-dual-general-estimate}. Then
\begin{equation}\label{eq:backward-uniform-estimate}
\begin{aligned}
 \norm{g(0)}_{\dbX}^2
 &+\norm{g}_{L^2(0,T;\dbX)}^2
 +\norm{\xi_-}_{L^2(0,T;\dbU)}^2\le C\left(
 \norm{g_{T}}_{\dbX}^2
 +\norm{q}_{L^2(0,T;\dbX)}^2
 +\norm{z}_{L^2(0,T;\dbY)}^2\right),
\end{aligned}
\end{equation}
where \(C\) is independent of \(T\).
\end{lemma}

\begin{proof}
By symmetry of \(E^{-1}A_j^\top\) and integration by
parts, we have
\begin{equation*}
\begin{aligned}
 -\frac{\dd}{\dd t}\mathcal E^*(g(t))
 &=\int_\Omega g^\top
 \left[
 QE^{-1}+E^{-1}Q^\top
 -\sum_{j=1}^d(E^{-1}A_j^\top)_{x_j}
 \right]g\dd x\\
 &\quad+\int_\Gamma g^\top E^{-1}A_\nu^\top g\dd\sigma
 +2\int_\Omega g^\top E^{-1}q\dd x.
\end{aligned}
\end{equation*}
Using \eqref{eq:inverse-internal-dissipation},
\eqref{eq:dualdissipation}, and Young's inequality, we have
\begin{equation}\label{eq:backward-complete-energy}
 -\frac{\dd}{\dd t}\mathcal E^*(g(t))
 +c\norm{g(t)}_{\dbX}^2
 +c\norm{\xi_-(t)}_{\dbU}^2
 \le C\left(
 \norm{q(t)}_{\dbX}^2+\norm{z(t)}_{\dbY}^2\right).
\end{equation}
Integrating \eqref{eq:backward-complete-energy} over \((0,T)\) and using the norm equivalence of $\mathcal{E}^*(g(t))$ 
yields
\eqref{eq:backward-uniform-estimate}. 
\end{proof}

Finally, we study the pairing \(\mathcal P[f(t),g(t)]:=\ip{f(t)}{g(t)}_{\dbX}\).
\begin{lemma}
\label{lem:dynamic-green-identity}
Let \(f\) solve \eqref{eq:evolution-system}, and let the pair \((g,\xi_-)\) solve
\eqref{eq:dynamic-dual-general-estimate}. Then
\begin{equation}\label{eq:cross-green-integrated}
\begin{aligned}
\mathcal P[f(0),g(0)]
 -\mathcal P[f(T),g(T)]=\int_0^T\bigl(
 \ip{f(t)}{q(t)}_{\dbX}
 +\ip{\zeta_+(t)}{z(t)}_{\dbY}
 +\ip{u(t)}{\Lambda_-\xi_-(t)}_{\dbU}
 \bigr)\dd t.
\end{aligned}
\end{equation}
\end{lemma}

\begin{proof}
Differentiating
\(\mathcal P[f(t),g(t)]\) and using the forward and adjoint equations
give
\[
 \frac{\dd}{\dd t}\mathcal P[f(t),g(t)]
 =-\int_\Gamma f^\top A_\nu^\top g\dd\sigma
  -\int_\Omega f^\top q\dd x,
\]
because the interior \(Q\)-terms cancel. By
\eqref{eq:standard-boundary-pairing},
\(\zeta_-=K\zeta_++u\), and
\(\Lambda_+\xi_++K^\top\Lambda_-\xi_-=z\), we have
\[
 f^\top A_\nu^\top g
 =\zeta_+^\top z+u^\top\Lambda_-\xi_-,
\]
and consequently
\begin{equation} \label{eq:dynamic-pairing}
  \frac{\dd}{\dd t}\mathcal P[f(t),g(t)]
 =-\bigl(
 \ip{f(t)}{q(t)}_{\dbX}
 +\ip{\zeta_+(t)}{z(t)}_{\dbY}
 +\ip{u(t)}{\Lambda_-\xi_-(t)}_{\dbU}
 \bigr).
\end{equation}
Integrating \eqref{eq:dynamic-pairing} over \((0,T)\) proves
\eqref{eq:cross-green-integrated}. 
\end{proof}

\subsection{Cost functions}
\label{subsec:cost-assumptions}
For the cost densities \(\Phi=\Phi(x,f): \Omega \times \R^n \to \R\) and \(F=F(x,u,\zeta_+): \Gamma\times \R^{n_-}\times \R^{n_+} \to \R\) in
\eqref{eq:dynamic-cost-overview} and
\eqref{eq:static-cost-overview}, we impose the following assumption.

\begin{assumption}[Convexity and regularity of the costs]
\label{ass:cost}
The function \(\Phi=\Phi(x,f)\) is continuous on
\(\overline\Omega\times\R^n\) and twice continuously differentiable
with respect to \(f\). On each planar face, $F=F(x,u,\zeta_+)$ is continuous
up to the closure of the face and twice continuously differentiable
with respect to $(u,\zeta_+)$. Furthermore, there exists a positive constant \(C\geq1\) such that
\begin{equation*}
 0\leq\Phi_{ff}(x,f)\leq CI_n,
 \qquad
 C^{-1}I_{n_-+n_+}
 \leq D^2_{(u,\zeta_+)}F(x,u,\zeta_+)
 \leq CI_{n_-+n_+}.
\end{equation*}
Here
\(\Phi_{ff}=D_f^2\Phi\) and \(D^2_{(u,\zeta_+)}F\) denote the Hessian
matrices with respect to \(f\) and \((u,\zeta_+)\), respectively.
\end{assumption}
Note that
this assumption naturally extends the classical LQR problem,
in which $F$ is quadratic and its Hessian is a constant positive definite
matrix.
Assumption~\ref{ass:cost} also gives the existence and
uniqueness of the minimizers in \eqref{eq:dynamic-cost-overview} and
\eqref{eq:static-cost-overview}. Indeed, for the dynamic problem \eqref{eq:dynamic-cost-overview},
Theorem~\ref{thm:wellposedness} and Lemma~\ref{lem:forward-energy-estimate} show that, for a fixed initial datum,
the map \(u\mapsto(f,\zeta_+)\) is affine and continuous from
\(L^2(0,T;\dbU)\) into
\(L^2(0,T;\dbX)\times L^2(0,T;\dbY)\). The upper Hessian bounds in
Assumption~\ref{ass:cost} give at most quadratic growth of the cost
densities, so the  functional \(J_T\) is finite and continuous.
Moreover, \(\Phi(x,\cdot)\) is convex and
\(F(x,\cdot,\cdot)\) is uniformly strongly convex. Hence \(J_T\) is
strongly convex in \(u\). It is therefore coercive and weakly lower
semicontinuous. The direct method gives a minimizer, and strong
convexity gives its uniqueness. Similarly,
Theorem~\ref{thm:static-wp} and the same argument give a unique
minimizer of \(J_s\).

\subsection{First-order optimality conditions}
\label{subsec:optimality-conditions}

We first derive the reduced first-order optimality conditions and then
represent the resulting Hilbert-adjoint terms by adjoint boundary
traces. For a fixed initial value \(f_0\in\dbX\), let \(f[u]\) and
\(\zeta_+[u]\) denote the state and observation generated by
\(u\in L^2(0,T;\dbU)\). Moreover, for the zero initial datum \(f_0=0\), we define the following maps
\begin{equation*}
\begin{aligned}
 \mathcal S_T:L^2(0,T;\dbU)&\longrightarrow L^2(0,T;\dbX),
 &\mathcal S_Tu&:=f,\\
 \mathcal G_T:L^2(0,T;\dbU)&\longrightarrow L^2(0,T;\dbY),
 &\mathcal G_Tu&:=\zeta_+.
\end{aligned}
\end{equation*}
Theorem~\ref{thm:wellposedness} shows that both maps are well-defined and bounded. We denote their Hilbert adjoints by
\begin{equation*}
\begin{aligned}
 \mathcal S_T^*:L^2(0,T;\dbX)&\longrightarrow L^2(0,T;\dbU),\qquad
 \mathcal G_T^*:L^2(0,T;\dbY)&\longrightarrow L^2(0,T;\dbU).
\end{aligned}
\end{equation*}
Given the initial datum $f_0$, we fix the dynamic minimizer \(u_T\). Let
\(h\in L^2(0,T;\dbU)\) and \(\varepsilon\in\R\) be arbitrary. By linearity,
\begin{equation*}
 f[u_T+\varepsilon h]=f_T+\varepsilon\mathcal S_Th,
 \qquad
 \zeta_+[u_T+\varepsilon h]=\zeta_{T,+}+\varepsilon\mathcal G_Th.
\end{equation*}

Substituting these identities into
\eqref{eq:dynamic-cost-overview} and applying Taylor's formula gives
\begin{equation*}
\begin{aligned}
 J_T(u_T+\varepsilon h)
 =&\int_0^T\bigg[
 \int_\Omega
 \Phi\bigl(x,f_T+\varepsilon\mathcal S_Th\bigr)\dd x
 +\int_\Gamma F\bigl(x,u_T+\varepsilon h,
 \zeta_{T,+}+\varepsilon\mathcal G_Th\bigr)\dd\sigma
 \bigg]\dd t\\
 =&J_T(u_T)+\varepsilon\int_0^T\bigg[
 \int_\Omega\Phi^\top_f(x,f_T)\mathcal S_Th\dd x \\
 &+\int_\Gamma\bigl(
 F^\top_U(x,u_T,\zeta_{T,+}) h
 +F^\top_{\zeta_+}(x,u_T,\zeta_{T,+})\mathcal G_Th
 \bigr)\dd\sigma
 \bigg]\dd t+O(\varepsilon^2).
\end{aligned}
\end{equation*}
We write \(\Phi_f\), \(F_U\), and \(F_{\zeta_+}\) for the gradients of
\(\Phi\) and \(F\) with respect to \(f\), \(u\), and \(\zeta_+\),
respectively. 
Since \(u_T\) is a minimizer,
\(J_T(u_T+\varepsilon h)-J_T(u_T)\geq0\) for every
\(\varepsilon\in\R\). Hence
\begin{equation*}
\begin{aligned}
 0
 &=\int_0^T\bigl[
 \ip{\Phi_f(f_T)}{\mathcal S_Th}_{\dbX}
 +\ip{F_U(u_T,\zeta_{T,+})}{h}_{\dbU}
 +\ip{F_{\zeta_+}(u_T,\zeta_{T,+})}{\mathcal G_Th}_{\dbY}
 \bigr]\dd t\\
 &=\ip{
 F_U(u_T,\zeta_{T,+})+\mathcal S_T^*\Phi_f(f_T)
 +\mathcal G_T^*F_{\zeta_+}(u_T,\zeta_{T,+})
 }{h}_{L^2(0,T;\dbU)}.
\end{aligned}
\end{equation*}
The upper Hessian bounds in Assumption~\ref{ass:cost} imply that
\(\Phi_f(f_T)\in L^2(0,T;\dbX)\),
\(F_U(u_T,\zeta_{T,+})\in L^2(0,T;\dbU)\) and
\(F_{\zeta_+}(u_T,\zeta_{T,+})\in L^2(0,T;\dbY)\). Since \(h\) is arbitrary, we obtain
\begin{equation}
\label{eq:dynamic-optimality}
 F_U(u_T,\zeta_{T,+})+\mathcal S_T^*\Phi_f(f_T)
 +\mathcal G_T^*F_{\zeta_+}(u_T,\zeta_{T,+})=0
 \qquad\text{in }L^2(0,T;\dbU).
\end{equation}

Similarly, Theorem~\ref{thm:static-wp} defines bounded maps
\begin{equation*}
 \mathcal S_s:\dbU\longrightarrow\dbX,
 \quad \mathcal S_su:=f,
 \qquad
 \mathcal G_s:\dbU\longrightarrow\dbY,
 \quad \mathcal G_su:=\zeta_+,
\end{equation*}
and the optimality condition for \eqref{eq:static-cost-overview} is
\begin{equation}
\label{eq:static-optimality}
 F_U(u_s,\zeta_{s,+})+\mathcal S_s^*\Phi_f(f_s)
 +\mathcal G_s^*F_{\zeta_+}(u_s,\zeta_{s,+})=0
 \qquad\text{in }\dbU.
\end{equation}

The optimality conditions involve the combinations
\(\mathcal S_T^*q+\mathcal G_T^*z\) and
\(\mathcal S_s^*q+\mathcal G_s^*z\). We now represent these  combinations in explicit form
 by adjoint boundary traces.

\begin{proposition}
\label{prop:dynamic-adjoint-representation}
Given \(q\in L^2(0,T;\dbX)\) and
\(z\in L^2(0,T;\dbY)\), let \(g\) solve
\eqref{eq:dynamic-dual-general-estimate} with zero terminal datum, and
let \(\xi_-\) be the outgoing variable of \(\xi=\Pi^\top g\). Then
\begin{equation}
\label{eq:combined-dynamic-adjoint}
 \mathcal S_T^*q+\mathcal G_T^*z
 =-\Lambda_-\xi_-
 \qquad\text{in }L^2(0,T;\dbU).
\end{equation}
Similarly, given \(q\in\dbX\) and \(z\in\dbY\), let
\(g^{\mathrm s}\) be the solution to
\eqref{eq:stationary-adjoint-system} and \(\xi_-^{\mathrm s}\) its
outgoing variable. Then
\begin{equation}
\label{eq:combined-static-adjoint}
 \mathcal S_s^*q+\mathcal G_s^*z=-\Lambda_-\xi_-^{\mathrm s}
 \qquad\text{in }\dbU.
\end{equation}
\end{proposition}

\begin{proof}
We first prove \eqref{eq:combined-dynamic-adjoint}. Let
\(h\in L^2(0,T;\dbU)\), and apply
\eqref{eq:cross-green-integrated} to
\(f=\mathcal S_Th\) and \(\zeta_+=\mathcal G_Th\). Since \(f(0)=0\) and
\(g(T)=0\), we obtain
\begin{equation*}
 \ip{\mathcal S_Th}{q}_{L^2(0,T;\dbX)}
 +\ip{\mathcal G_Th}{z}_{L^2(0,T;\dbY)}
 =\ip{h}{-\Lambda_-\xi_-}_{L^2(0,T;\dbU)}.
\end{equation*}
The definitions of the Hilbert adjoints give
\eqref{eq:combined-dynamic-adjoint}. Similarly,
we have
\begin{equation*}
 \ip{\mathcal S_sh}{q}_{\dbX}+\ip{\mathcal G_sh}{z}_{\dbY}
 =\ip{h}{-\Lambda_-\xi_-^{\mathrm s}}_{\dbU},
 \qquad h\in\dbU,
\end{equation*}
which proves \eqref{eq:combined-static-adjoint}.
\end{proof}

We finally apply these representations to
\eqref{eq:dynamic-optimality} and \eqref{eq:static-optimality}. We write \((g,\xi_-)[q,z]\), where \(g\) is the solution to
\eqref{eq:dynamic-dual-general-estimate} with interior datum \(q\),
boundary datum \(z\), and zero terminal datum, and \(\xi_-\) is its
outgoing variable. Similarly, we write
\((g^{\mathrm s},\xi_-^{\mathrm s})[q,z]\), where
\(g^{\mathrm s}\) is the solution to
\eqref{eq:stationary-adjoint-system} and \(\xi_-^{\mathrm s}\) is its
outgoing variable. Define
\begin{equation}
\label{eq:gT-gs-def}
\begin{aligned}
 (g_T,\xi_{T,-}):=(g,\xi_-)\bigl[\Phi_f(f_T),F_{\zeta_+}(u_T,\zeta_{T,+})\bigr],\quad
 (g_s,\xi_{s,-})
 :=(g^{\mathrm s},\xi_-^{\mathrm s})\bigl[\Phi_f(f_s),F_{\zeta_+}(u_s,\zeta_{s,+})\bigr].
\end{aligned}
\end{equation}

Equation~\eqref{eq:combined-dynamic-adjoint} gives
\begin{equation}
\label{eq:dynamic-boundary-optimality}
 F_U(u_T,\zeta_{T,+})-\Lambda_-\xi_{T,-}=0
 \qquad\text{in }L^2(0,T;\dbU).
\end{equation}
Similarly, \eqref{eq:combined-static-adjoint} gives
\begin{equation}
\label{eq:static-boundary-optimality}
 F_U(u_s,\zeta_{s,+})-\Lambda_-\xi_{s,-}=0
 \qquad\text{in }\dbU.
\end{equation}

\section{Proof of the exponential turnpike}
\label{exp:sec:exponential-turnpike}
The boundary conditions, energy estimates, and optimality conditions
established in Section~\ref{sec:preliminaries} provide the main
ingredients for the turnpike analysis. We now use them to prove the
exponential turnpike property for the dynamic optimal control problem
\eqref{eq:dynamic-cost-overview} under
Assumptions~\ref{ass:stabilizing-symmetrizer},
\ref{ass:boundary-dissipativity}, and \ref{ass:cost}.
To compare the dynamic and static optimal solutions, we introduce the
following deviations:
\[
 \widetilde u:=u_T-u_s,
 \qquad
 \widetilde\zeta_+:=\zeta_{T,+}-\zeta_{s,+},
 \qquad
 \widetilde f:=f_T-f_s.
\]
Subtracting the stationary state system \eqref{eq:stationary-system}
from the dynamic state system \eqref{eq:evolution-system}, we find that,
with \(\widetilde\zeta:=\Pi^{-1}\widetilde f\), the forward deviation
satisfies
\begin{equation}\label{eq:deviation-forward}
\begin{cases}
 \displaystyle
 \widetilde f_t+\sum_{j=1}^dA_j\widetilde f_{x_j}
 =Q\widetilde f,\\[1mm]
 \widetilde\zeta_-=K\widetilde\zeta_++\widetilde u,
 \qquad
 \widetilde y=\widetilde\zeta_+,\\[1mm]
 \widetilde f(0)=f_0-f_s.
\end{cases}
\end{equation}

To represent the adjoint term in \eqref{eq:dynamic-optimality}, we take
\(q=\Phi_f(x,f_T)\) and
\(z=F_{\zeta_+}(x,u_T,\zeta_{T,+})\) in
\eqref{eq:dynamic-dual-general-estimate}, with zero terminal datum.
Similarly, for \eqref{eq:static-optimality}, we take
\(q=\Phi_f(x,f_s)\) and \(z=F_{\zeta_+}(x,u_s,\zeta_{s,+})\) in
\eqref{eq:stationary-adjoint-system}. These are precisely the quantities
defined in \eqref{eq:gT-gs-def}.

Set \(\widetilde g:=g_T-g_s\) and
\(\widetilde\xi_-:=\xi_{T,-}-\xi_{s,-}\), and write
\[
\begin{aligned}
 \widetilde\Phi_f&:=\Phi_f(x,f_T)-\Phi_f(x,f_s),\\
 \widetilde F_U&:=F_U(x,u_T,\zeta_{T,+})-F_U(x,u_s,\zeta_{s,+}),\\
 \widetilde F_{\zeta_+}&:=F_{\zeta_+}(x,u_T,\zeta_{T,+})-F_{\zeta_+}(x,u_s,\zeta_{s,+}).
\end{aligned}
\]
Subtracting the stationary adjoint system
\eqref{eq:stationary-adjoint-system} from the dynamic adjoint system
\eqref{eq:dynamic-dual-general-estimate} gives
\begin{equation}\label{eq:deviation-backward}
\begin{cases}
 \displaystyle
 \widetilde g_t
 +\sum_{j=1}^dA_j^\top\widetilde g_{x_j}
 =-Q^\top\widetilde g-\widetilde\Phi_f,\\[1mm]
 \Lambda_+\widetilde\xi_+
 +K^\top\Lambda_-\widetilde\xi_-
 =\widetilde F_{\zeta_+},\\[1mm]
 \widetilde g(T)=-g_s.
\end{cases}
\end{equation}
Finally, subtracting \eqref{eq:static-boundary-optimality} from
\eqref{eq:dynamic-boundary-optimality} gives
\begin{equation}\label{eq:stationarity-deviation}
 \widetilde F_U-\Lambda_-\widetilde\xi_-=0.
\end{equation}
In \eqref{eq:deviation-forward}--\eqref{eq:deviation-backward} and
throughout this section, static quantities are understood as
time-independent functions on \((0,T)\).

Now we introduce the time-weight
\begin{equation*}
 d_T(t):=\min\{t,T-t\},
 \qquad
 \ThetaT(t):=e^{2\mu d_T(t)}.
\end{equation*}
Then \(\ThetaT\in W^{1,\infty}(0,T)\),
\(\ThetaT(0)=\ThetaT(T)=1\), and
\begin{equation}\label{eq:Theta-derivative}
 \abs{\ThetaT'(t)}\le2\mu\ThetaT(t)
 \qquad\text{for a.e. }t\in(0,T).
\end{equation}

\begin{remark}
The proof of Proposition~\ref{prop:weighted-state-io-turnpike} uses
integration by parts with \(\ThetaT\).
The corner of \(\ThetaT\) at \(t=T/2\) causes no difficulty because
\(\ThetaT\in W^{1,\infty}(0,T)\); see
\cite[Corollary~8.10, p.~215]{Brezis2011}.
\end{remark}

We have the following technical estimate.
\begin{proposition}\label{prop:weighted-state-io-turnpike}
Under Assumptions~\ref{ass:stabilizing-symmetrizer},
\ref{ass:boundary-dissipativity}, and \ref{ass:cost}, there exist
\(C,\mu>0\), independent of \(T\), such that
\begin{equation}\label{eq:weighted-state-io-turnpike}
 \int_0^T e^{2\mu d_T(t)}
 \left(
 \norm{\widetilde f(t)}_{\dbX}^2
 +\norm{\widetilde u(t)}_{\dbU}^2
 +\norm{\widetilde\zeta_+(t)}_{\dbY}^2
 \right)\dd t
 \le C.
\end{equation}
\end{proposition}

\begin{proof}
Set
\begin{equation*}
 \mathcal I
 :=\int_0^T\ThetaT(t)
 \left(\norm{\widetilde u(t)}_{\dbU}^2
 +\norm{\widetilde\zeta_+(t)}_{\dbY}^2
 \right)\dd t.
\end{equation*}
We first derive the weighted forward estimate. Applying
\eqref{eq:forward-complete-energy} to \(\widetilde f\) in \eqref{eq:deviation-forward}, multiplying by
\(\ThetaT\), and integrating by parts give
\begin{equation*}
\begin{aligned}
 \mathcal E(\widetilde f(T))
 -\mathcal E(\widetilde f(0))
 -\int_0^T\ThetaT'\mathcal E(\widetilde f)\dd t
 +c\int_0^T\ThetaT\norm{\widetilde f}_{\dbX}^2\dd t
 +c\int_0^T\ThetaT\norm{\widetilde\zeta_+}_{\dbY}^2\dd t
 \le C\int_0^T\ThetaT\norm{\widetilde u}_{\dbU}^2\dd t.
\end{aligned}
\end{equation*}
By \eqref{eq:Theta-derivative} and the equivalence of \(\mathcal E\)
with the squared \(\dbX\)-norm, there exists \(\mu_0>0\), independent
of \(T\), such that, for every \(\mu\in(0,\mu_0)\), the weight term is
absorbed in the sense that
\begin{equation*}
 -\int_0^T\ThetaT'(t)\mathcal E(\widetilde f(t))\dd t
 +c\int_0^T\ThetaT\norm{\widetilde f}_{\dbX}^2\dd t
 \ge\frac c2\int_0^T\ThetaT\norm{\widetilde f}_{\dbX}^2\dd t.
\end{equation*}
Consequently, after changing \(C\),
\begin{equation}\label{eq:weighted-forward-deviation}
\begin{aligned}
 \mathcal E(\widetilde f(T))
 &+\int_0^T\ThetaT\norm{\widetilde f}_{\dbX}^2\dd t
 +\int_0^T\ThetaT\norm{\widetilde\zeta_+}_{\dbY}^2\dd t
 \le C\left(\mathcal E(\widetilde f(0))+\mathcal I\right).
\end{aligned}
\end{equation}

Similarly, after decreasing \(\mu_0\) if necessary, applying the same
weighted argument to \eqref{eq:backward-complete-energy}, with
\(q=\widetilde\Phi_f\), \(z=\widetilde F_{\zeta_+}\), and
\(\widetilde g(T)=-g_s\), gives, for every \(\mu\in(0,\mu_0)\),
\begin{equation*}
\begin{aligned}
 \mathcal E^*(\widetilde g(0))
 &+\int_0^T\ThetaT\norm{\widetilde g}_{\dbX}^2\dd t
 \le C\biggl(
 \mathcal E^*(g_s)
 +\int_0^T\ThetaT\left(
 \norm{\widetilde\Phi_f}_{\dbX}^2
 +\norm{\widetilde F_{\zeta_+}}_{\dbY}^2\right)\dd t
 \biggr).
\end{aligned}
\end{equation*}

Set \(D_0:=\mathcal E(\widetilde f(0))+\mathcal E^*(g_s)\), which is
independent of \(T\). The upper Hessian bounds in
Assumption~\ref{ass:cost} and \eqref{eq:weighted-forward-deviation} give
\begin{equation*}
\begin{aligned}
 \int_0^T\ThetaT\norm{\widetilde\Phi_f}_{\dbX}^2\dd t
 &\le C\int_0^T\ThetaT\norm{\widetilde f}_{\dbX}^2\dd t
 \le C\left(D_0+\mathcal I\right),\\
 \int_0^T\ThetaT\norm{\widetilde F_{\zeta_+}}_{\dbY}^2\dd t
 &\le C\mathcal I.
\end{aligned}
\end{equation*}
Substituting these bounds into the weighted adjoint estimate yields
\begin{equation}\label{eq:weighted-backward-deviation}
\begin{aligned}
 \mathcal E^*(\widetilde g(0))
 +\int_0^T\ThetaT\norm{\widetilde g}_{\dbX}^2\dd t
 &\le C\left(D_0+\mathcal I\right).
\end{aligned}
\end{equation}

To close the preceding estimates, we use the calculation in the proof
of Lemma~\ref{lem:dynamic-green-identity}, together with
\eqref{eq:stationarity-deviation} and Assumption~\ref{ass:cost}, to obtain
\begin{equation*}
\begin{aligned}
 -\frac{\dd}{\dd t}
 \mathcal P[\widetilde f(t),\widetilde g(t)]
 &=\ip{\widetilde f}{\widetilde\Phi_f}_{\dbX}
 +\ip{\widetilde\zeta_+}{\widetilde F_{\zeta_+}}_{\dbY}
 +\ip{\widetilde u}{\Lambda_-\widetilde\xi_-}_{\dbU}\\
 &=\ip{\widetilde f}{\widetilde\Phi_f}_{\dbX}
 +\ip{\widetilde u}{\widetilde F_U}_{\dbU}
 +\ip{\widetilde\zeta_+}{\widetilde F_{\zeta_+}}_{\dbY}\\
 &\ge C^{-1}\left(
 \norm{\widetilde u}_{\dbU}^2
 +\norm{\widetilde\zeta_+}_{\dbY}^2\right).
\end{aligned}
\end{equation*}
Multiplying this inequality by \(\ThetaT\) and integrating by parts
give
\begin{equation*}
 \mathcal I
 \le C\biggl(
 \abs{\mathcal P[\widetilde f(0),\widetilde g(0)]}
 +\abs{\mathcal P[\widetilde f(T),\widetilde g(T)]}
 +\abs{\int_0^T\ThetaT'(t)
 \mathcal P[\widetilde f(t),\widetilde g(t)]\dd t}\biggr).
\end{equation*}
For every \(\varepsilon>0\), we use
\eqref{eq:weighted-forward-deviation} to control
\(\mathcal E(\widetilde f(T))\) and
\eqref{eq:weighted-backward-deviation} to control
\(\mathcal E^*(\widetilde g(0))\). Together with the endpoint data
\(\widetilde f(0)=f_0-f_s\) and \(\widetilde g(T)=-g_s\), the energy
equivalence, and Young's inequality, this gives
\begin{equation*}
 \abs{\mathcal P[\widetilde f(0),\widetilde g(0)]}
 +\abs{\mathcal P[\widetilde f(T),\widetilde g(T)]}
 \le C_\varepsilon D_0
 +\varepsilon C\mathcal I.
\end{equation*}
Using \eqref{eq:weighted-forward-deviation} and
\eqref{eq:weighted-backward-deviation} once more, together with
\eqref{eq:Theta-derivative}, we obtain
\begin{equation*}
\begin{aligned}
 \abs{\int_0^T\ThetaT'(t)
 \mathcal P[\widetilde f(t),\widetilde g(t)]\dd t}
 \le C\mu\int_0^T\ThetaT
 \norm{\widetilde f}_{\dbX}\norm{\widetilde g}_{\dbX}\dd t\le C\mu\left(D_0+\mathcal I\right).
\end{aligned}
\end{equation*}
Combining these inequalities gives
\begin{equation*}
 \mathcal I
 \le C_\varepsilon D_0
 +C(\varepsilon+\mu)\left(D_0+\mathcal I\right).
\end{equation*}
Choose \(\varepsilon>0\) so that \(C\varepsilon\le1/4\). After
decreasing \(\mu_0\) once more if necessary, we may assume that
\(C\mu_0\le1/4\). Hence, for every \(\mu\in(0,\mu_0)\),
\begin{equation*}
 \mathcal I\le CD_0.
\end{equation*}
Combining this estimate with \eqref{eq:weighted-forward-deviation}
proves \eqref{eq:weighted-state-io-turnpike}, since \(D_0\) is
independent of \(T\).
\end{proof}

\begin{theorem}[Exponential turnpike]\label{thm:main-exponential}
Under Assumptions~\ref{ass:stabilizing-symmetrizer},
\ref{ass:boundary-dissipativity}, and \ref{ass:cost}, there exist
constants \(C,\mu>0\), independent of \(T\), such that, with
\(I_t:=\bigl(d_T(t),T-d_T(t)\bigr)\),
\begin{equation}\label{eq:central-interval-turnpike}
 \norm{f_T(t)-f_s}_{\dbX}
 +\norm{u_T-u_s}_{L^2(I_t;\dbU)}
 +\norm{\zeta_{T,+}-\zeta_{s,+}}_{L^2(I_t;\dbY)}
 \le C\left(e^{-\mu t}+e^{-\mu(T-t)}\right),
 \qquad 0\le t\le T.
\end{equation}
In particular, for every \(t\in[0,T/2]\),
\begin{equation*}
  \norm{u_T-u_s}_{L^2((t,T-t);\dbU)}
  +\norm{\zeta_{T,+}-\zeta_{s,+}}_{L^2((t,T-t);\dbY)}
  \le Ce^{-\mu t}.
\end{equation*}
\end{theorem}

\begin{proof}
Write \(\widetilde f:=f_T-f_s\), \(\widetilde u:=u_T-u_s\), and
\(\widetilde\zeta_+:=\zeta_{T,+}-\zeta_{s,+}\). Applying
\eqref{eq:forward-complete-energy} to \(\widetilde f\) in
\eqref{eq:deviation-forward},
we obtain from the energy equivalence an \(\alpha>0\) such that
\begin{equation*}
 \frac{\dd}{\dd t}\mathcal E(\widetilde f(t))
 +\alpha\mathcal E(\widetilde f(t))
 \le C\norm{\widetilde u(t)}_{\dbU}^2.
\end{equation*}
Therefore
\begin{equation}\label{eq:state-convolution}
 \mathcal E(\widetilde f(t))
 \le e^{-\alpha t}\mathcal E(\widetilde f(0))
 +C\int_0^t e^{-\alpha(t-s)}
 \norm{\widetilde u(s)}_{\dbU}^2\dd s.
\end{equation}
After decreasing \(\mu\) if necessary, assume \(2\mu<\alpha\). Since
\(d_T\) is \(1\)-Lipschitz,
\begin{equation*}
 t-s\ge d_T(t)-d_T(s).
\end{equation*}
Consequently,
\begin{equation*}
 e^{-\alpha(t-s)}\norm{\widetilde u(s)}_{\dbU}^2
 \le e^{-2\mu(t-s)}\norm{\widetilde u(s)}_{\dbU}^2
 \le e^{-2\mu d_T(t)}\ThetaT(s)\norm{\widetilde u(s)}_{\dbU}^2.
\end{equation*}
Also \(d_T(t)\le t\) and \(2\mu<\alpha\), so
\(e^{-\alpha t}\le e^{-2\mu d_T(t)}\).
Using \eqref{eq:weighted-state-io-turnpike} in
\eqref{eq:state-convolution} gives
\begin{equation*}
 \mathcal E(\widetilde f(t))\le Ce^{-2\mu d_T(t)}.
\end{equation*}
Hence, by the energy equivalence,
\begin{equation*}
 \norm{\widetilde f(t)}_{\dbX}
 \le Ce^{-\mu d_T(t)}
 \le C\left(e^{-\mu t}+e^{-\mu(T-t)}\right).
\end{equation*}

For fixed \(t\in[0,T]\), every \(s\in I_t\) satisfies
\(d_T(s)\ge d_T(t)\), and hence
\(\ThetaT(s)\ge e^{2\mu d_T(t)}\). Therefore,
\eqref{eq:weighted-state-io-turnpike} gives
\begin{equation*}
\begin{aligned}
 e^{2\mu d_T(t)}\int_{I_t}
 \left(\norm{\widetilde u(s)}_{\dbU}^2
 +\norm{\widetilde\zeta_+(s)}_{\dbY}^2\right)\dd s
 &\le\int_{I_t}\ThetaT(s)
 \left(\norm{\widetilde u(s)}_{\dbU}^2
 +\norm{\widetilde\zeta_+(s)}_{\dbY}^2\right)\dd s\\
 &\le\int_0^T\ThetaT(s)
 \left(\norm{\widetilde u(s)}_{\dbU}^2
 +\norm{\widetilde\zeta_+(s)}_{\dbY}^2\right)\dd s\le C.
\end{aligned}
\end{equation*}
Dividing by \(e^{2\mu d_T(t)}\) and taking square roots gives
\begin{equation*}
 \norm{\widetilde u}_{L^2(I_t;\dbU)}
 +\norm{\widetilde\zeta_+}_{L^2(I_t;\dbY)}
 \le Ce^{-\mu d_T(t)}.
\end{equation*}
Together with the preceding pointwise estimate, this proves
\eqref{eq:central-interval-turnpike}. Finally, if
\(t\in[0,T/2]\), then \(d_T(t)=t\), \(I_t=(t,T-t)\), and
\(e^{-\mu(T-t)}\le e^{-\mu t}\). Hence
\begin{equation*}
 \norm{\widetilde u}_{L^2((t,T-t);\dbU)}
 +\norm{\widetilde\zeta_+}_{L^2((t,T-t);\dbY)}
 \le Ce^{-\mu t},
\end{equation*}
which proves the second estimate.
\end{proof}

The weighted estimate underlying
Theorem~\ref{thm:main-exponential} also yields the corresponding
integral and measure turnpike properties.
\begin{corollary}[Integral and measure turnpikes]
\label{cor:integral-measure}
Under the assumptions of Theorem~\ref{thm:main-exponential},
\begin{equation}\label{eq:integral-corollary}
 \int_0^T\left(
 \norm{\widetilde f(t)}_{\dbX}^2
 +\norm{\widetilde u(t)}_{\dbU}^2
 +\norm{\widetilde\zeta_+(t)}_{\dbY}^2
 \right)\dd t\le C,
\end{equation}
and the corresponding time average is bounded by \(C/T\).
Moreover, given \(\varepsilon_0>0\), define
\begin{equation*}
 \mathscr M_{\varepsilon_0}^T
 :=\left\{t\in(0,T):
 \norm{\widetilde f(t)}_{\dbX}^2
 +\norm{\widetilde u(t)}_{\dbU}^2
 +\norm{\widetilde\zeta_+(t)}_{\dbY}^2>\varepsilon_0^2\right\}.
\end{equation*}
Then
\begin{equation}\label{eq:measure-corollary}
 \abs{\mathscr M_{\varepsilon_0}^T}\le\frac C{\varepsilon_0^2},
 \qquad
 \frac{\abs{\mathscr M_{\varepsilon_0}^T}}T
 \le\frac C{\varepsilon_0^2T}.
\end{equation}
\end{corollary}

\begin{proof}
The weight $\ThetaT$ in \eqref{eq:weighted-state-io-turnpike} is at least one,
which proves \eqref{eq:integral-corollary}. On
\(\mathscr M_{\varepsilon_0}^T\), its integrand is at least
\(\varepsilon_0^2\), and Chebyshev's inequality gives
\eqref{eq:measure-corollary}.
\end{proof}

So far, we have assumed that both the control and observation are
available on the entire boundary \(\Gamma\). In applications, however,
they may be available only on a portion of the boundary, as in the
Saint--Venant system considered in Section~\ref{sec:saint-venant}. We
therefore state the following localized version of
Theorem~\ref{thm:main-exponential}.
\begin{corollary}[Localized boundary actuation and observation]
\label{cor:partial-boundary}
Let \(\Gamma_c\subset\Gamma\) be measurable, and suppose that the
boundary input is supported on \(\Gamma_c\). Thus, in both the dynamic
and static problems, the boundary condition is
\begin{equation*}
 \zeta_-=K\zeta_+
 \quad\text{on }\Gamma\setminus\Gamma_c,
 \qquad
 \zeta_-=K\zeta_++u
 \quad\text{on }\Gamma_c.
\end{equation*}
Assume that, for some \(\delta>0\),
\begin{equation*}
 M-K^\top NK\geq\delta I_{n_+}
 \quad\text{on }\Gamma_c,
 \qquad
 M-K^\top NK\geq0
 \quad\text{on }\Gamma\setminus\Gamma_c.
\end{equation*}
Consider the cost functionals
\begin{equation*}
\begin{aligned}
 J_T(u)
 &=
 \int_0^T\left[
   \int_\Omega \Phi\bigl(x,f(t,x)\bigr)\dd x
   +\int_{\Gamma_c}
      F\bigl(x,u(t,x),\zeta_+(t,x)\bigr)\dd\sigma
 \right]\dd t,\\
 J_s(u)
 &=\int_\Omega \Phi\bigl(x,f(x)\bigr)\dd x
 +\int_{\Gamma_c}F\bigl(x,u(x),\zeta_+(x)\bigr)\dd\sigma.
\end{aligned}
\end{equation*}
Suppose that \(\Phi\) and the restriction of \(F\) to \(\Gamma_c\)
satisfy Assumption~\ref{ass:cost}, and that
Assumption~\ref{ass:stabilizing-symmetrizer} holds.
Set \(\widetilde f:=f_T-f_s\), \(\widetilde u:=u_T-u_s\), and
\(\widetilde\zeta_+:=\zeta_{T,+}-\zeta_{s,+}\).
Then there exist \(C,\mu>0\), independent of \(T\), such that, for
every \(t\in[0,T]\),
\begin{equation*}
\begin{aligned}
 \norm{\widetilde f(t)}_{\dbX}
 &+\left(\int_{d_T(t)}^{T-d_T(t)}
 \left[
 \norm{\widetilde u(s)}_{L^2(\Gamma_c;\R^{n_-})}^2
 +\norm{\widetilde\zeta_+(s)}_{L^2(\Gamma_c;\R^{n_+})}^2
 \right]\dd s\right)^{1/2}
 \le C\left(e^{-\mu t}+e^{-\mu(T-t)}\right).
\end{aligned}
\end{equation*}
\end{corollary}

The proof is omitted, since one can repeat the proofs of
Proposition~\ref{prop:weighted-state-io-turnpike} and
Theorem~\ref{thm:main-exponential}, restricting the boundary integrals
to \(\Gamma_c\) and discarding the nonnegative boundary contributions
on \(\Gamma\setminus\Gamma_c\).

\section{An example for the 2-D Saint--Venant equations}
\label{sec:saint-venant}

We now illustrate the preceding exponential turnpike result for the linearized
Saint--Venant equations and then examine their turnpike structure
numerically.

\subsection{Physical setting and Saint--Venant model}

In a physical river reach, continuous embankments are constructed along
the two banks.  They prevent water from escaping laterally, so the
velocity normal to either bank is zero.
Sluice gates are installed at the upstream and downstream ends.  By
regulating the water discharge, each gate prescribes a relation
between the water depth and the normal velocity.  This physical
configuration is shown in Figure~\ref{fig:sv-configuration}(a).

For simplicity, we model the river channel as the rectangle
\[
 \Omega=(0,L)\times(0,1),
\]
where \(L>0\) is the distance between the two gates.  As indicated in
Figure~\ref{fig:sv-configuration}(b), \(x_1=0\) and \(x_1=L\) represent
the upstream and downstream gates, whereas \(x_2=0\) and \(x_2=1\)
represent the solid embankments.

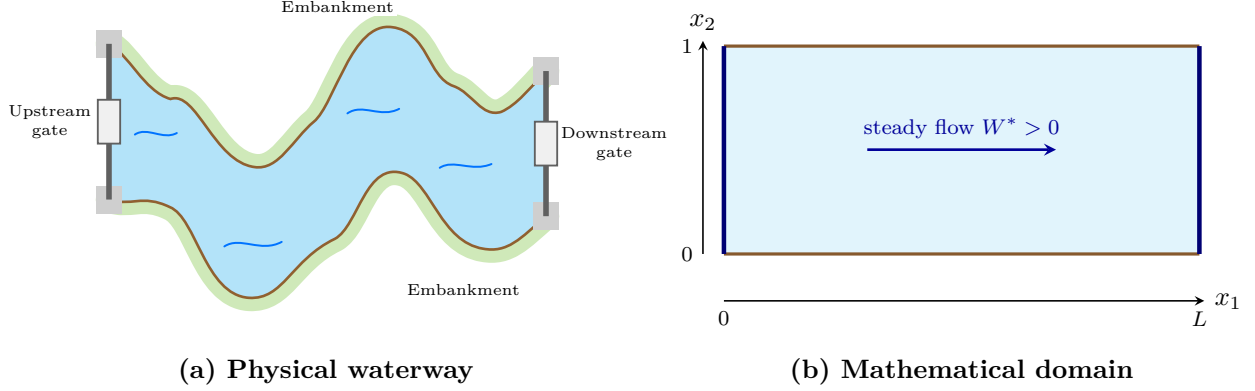
\begin{figure}[ht]
\centering
\begin{tikzpicture}[x=0.925cm,y=1cm,font=\small,>=stealth]
  \begin{scope}[xshift=-0.55cm]

  \draw[brown!55!green!32,line width=10pt,line cap=round,line join=round]
    (0.55,4.08)
    .. controls (0.82,3.96) and (1.05,3.45) .. (1.43,3.35)
    .. controls (1.80,3.47) and (2.05,2.52) .. (2.63,2.45)
    .. controls (2.98,2.40) and (3.18,2.82) .. (3.48,3.05)
    .. controls (3.78,3.28) and (4.02,4.38) .. (4.60,4.30)
    .. controls (5.12,4.24) and (5.10,3.64) .. (5.48,3.52)
    .. controls (5.75,3.43) and (5.88,3.10) .. (6.16,3.18)
    .. controls (6.38,3.27) and (6.58,3.50) .. (6.80,3.72);
  \draw[brown!55!green!32,line width=10pt,line cap=round,line join=round]
    (0.55,2.02)
    .. controls (0.86,1.92) and (1.02,2.10) .. (1.42,1.92)
    .. controls (1.80,1.78) and (1.92,0.76) .. (2.55,0.72)
    .. controls (3.05,0.69) and (3.15,1.22) .. (3.78,1.50)
    .. controls (4.12,1.67) and (4.22,2.55) .. (4.78,2.36)
    .. controls (5.16,2.22) and (5.30,1.46) .. (5.90,1.37)
    .. controls (6.25,1.31) and (6.43,1.50) .. (6.80,1.80);

  \path[fill=cyan!27]
    (0.55,4.08)
    .. controls (0.82,3.96) and (1.05,3.45) .. (1.43,3.35)
    .. controls (1.80,3.47) and (2.05,2.52) .. (2.63,2.45)
    .. controls (2.98,2.40) and (3.18,2.82) .. (3.48,3.05)
    .. controls (3.78,3.28) and (4.02,4.38) .. (4.60,4.30)
    .. controls (5.12,4.24) and (5.10,3.64) .. (5.48,3.52)
    .. controls (5.75,3.43) and (5.88,3.10) .. (6.16,3.18)
    .. controls (6.38,3.27) and (6.58,3.50) .. (6.80,3.72)
    -- (6.80,1.80)
    .. controls (6.43,1.50) and (6.25,1.31) .. (5.90,1.37)
    .. controls (5.30,1.46) and (5.16,2.22) .. (4.78,2.36)
    .. controls (4.22,2.55) and (4.12,1.67) .. (3.78,1.50)
    .. controls (3.15,1.22) and (3.05,0.69) .. (2.55,0.72)
    .. controls (1.92,0.76) and (1.80,1.78) .. (1.42,1.92)
    .. controls (1.02,2.10) and (0.86,1.92) .. (0.55,2.02)
    -- cycle;
  \draw[brown!75!black,line width=1.0pt]
    (0.55,4.08)
    .. controls (0.82,3.96) and (1.05,3.45) .. (1.43,3.35)
    .. controls (1.80,3.47) and (2.05,2.52) .. (2.63,2.45)
    .. controls (2.98,2.40) and (3.18,2.82) .. (3.48,3.05)
    .. controls (3.78,3.28) and (4.02,4.38) .. (4.60,4.30)
    .. controls (5.12,4.24) and (5.10,3.64) .. (5.48,3.52)
    .. controls (5.75,3.43) and (5.88,3.10) .. (6.16,3.18)
    .. controls (6.38,3.27) and (6.58,3.50) .. (6.80,3.72);
  \draw[brown!75!black,line width=1.0pt]
    (0.55,2.02)
    .. controls (0.86,1.92) and (1.02,2.10) .. (1.42,1.92)
    .. controls (1.80,1.78) and (1.92,0.76) .. (2.55,0.72)
    .. controls (3.05,0.69) and (3.15,1.22) .. (3.78,1.50)
    .. controls (4.12,1.67) and (4.22,2.55) .. (4.78,2.36)
    .. controls (5.16,2.22) and (5.30,1.46) .. (5.90,1.37)
    .. controls (6.25,1.31) and (6.43,1.50) .. (6.80,1.80);

  \draw[blue!55!cyan,line width=0.75pt,line cap=round]
    (0.92,2.88)..controls(1.10,2.99)and(1.32,2.80)..(1.52,2.90);
  \draw[blue!55!cyan,line width=0.75pt,line cap=round]
    (2.20,1.40)..controls(2.43,1.54)and(2.72,1.31)..(3.02,1.46);
  \draw[blue!55!cyan,line width=0.75pt,line cap=round]
    (3.96,3.16)..controls(4.18,3.30)and(4.43,3.09)..(4.70,3.22);
  \draw[blue!55!cyan,line width=0.75pt,line cap=round]
    (5.28,2.44)..controls(5.48,2.56)and(5.74,2.38)..(6.02,2.49);

  \fill[gray!38] (0.37,1.84) rectangle (0.73,2.20);
  \fill[gray!38] (0.37,3.90) rectangle (0.73,4.26);
  \draw[gray!75!black,line width=2.1pt] (0.55,2.02)--(0.55,4.08);
  \filldraw[fill=gray!12,draw=gray!75!black,line width=0.7pt]
    (0.39,2.76) rectangle (0.71,3.34);
  \fill[gray!38] (6.62,1.62) rectangle (6.98,1.98);
  \fill[gray!38] (6.62,3.54) rectangle (6.98,3.90);
  \draw[gray!75!black,line width=2.1pt] (6.80,1.80)--(6.80,3.72);
  \filldraw[fill=gray!12,draw=gray!75!black,line width=0.7pt]
    (6.64,2.47) rectangle (6.96,3.05);

  \node[font=\tiny,fill=white,inner sep=1.2pt] at (3.82,4.57)
    {Embankment};
  \node[font=\tiny,fill=white,inner sep=1.2pt] at (5.62,0.83)
    {Embankment};
  \node[font=\tiny,align=center,anchor=east,inner sep=0pt]
    at (0.30,3.05)
    {Upstream\\gate};
  \node[font=\tiny,align=center,anchor=west,inner sep=0pt]
    at (7.01,2.76)
    {Downstream\\gate};
  \node[font=\small\bfseries] at (3.65,-0.25)
    {(a) Physical waterway};
  \end{scope}

  \fill[cyan!10] (8.75,1.30) rectangle (15.55,4.05);
  \draw[brown!70!black,line width=1.2pt]
    (8.75,1.30)--(15.55,1.30);
  \draw[brown!70!black,line width=1.2pt]
    (8.75,4.05)--(15.55,4.05);
  \draw[blue!45!black,line width=1.7pt]
    (8.75,1.30)--(8.75,4.05);
  \draw[blue!45!black,line width=1.7pt]
    (15.55,1.30)--(15.55,4.05);

  \draw[->,blue!60!black,line width=1pt]
    (10.80,2.68)--(13.50,2.68)
    node[midway,above,font=\scriptsize] {steady flow $W^*>0$};

  \draw[->,line width=0.6pt] (8.75,0.68)--(15.62,0.68)
    node[right] {$x_1$};
  \draw[->,line width=0.6pt] (8.45,1.30)--(8.45,4.12)
    node[above] {$x_2$};
  \node[below,font=\scriptsize] at (8.75,0.68) {$0$};
  \node[below,font=\scriptsize] at (15.55,0.68) {$L$};
  \node[left,font=\scriptsize] at (8.45,1.30) {$0$};
  \node[left,font=\scriptsize] at (8.45,4.05) {$1$};

  \node[font=\small\bfseries] at (12.15,-0.25)
    {(b) Mathematical domain};
\end{tikzpicture}
\caption{Physical sluice-control configuration and its rectangular
mathematical idealization.}
\label{fig:sv-configuration}
\end{figure}

Let \(H=H(t,x_1,x_2)\) be the water depth,
\(W=W(t,x_1,x_2)\) the velocity along the channel, and
\(V=V(t,x_1,x_2)\) the velocity across the channel.  With bottom slopes
\(S_1,S_2\), linear drag coefficient \(k>0\), Coriolis coefficient
\(\ell>0\), and gravitational constant \(g>0\), the flow is governed by
the 2-D Saint--Venant equations \cite{YangYong2024}
\begin{equation}\label{eq:svquasi}
\begin{cases}
 H_t+WH_{x_1}+HW_{x_1}+VH_{x_2}+HV_{x_2}=0,\\[1mm]
 W_t+WW_{x_1}+VW_{x_2}+gH_{x_1}=gS_1-kW+\ell V,\\[1mm]
 V_t+WV_{x_1}+VV_{x_2}+gH_{x_2}=gS_2-\ell W-kV,
\end{cases}
\end{equation}
which is a first-order quasilinear hyperbolic system.

We first assume that the Saint--Venant equations admit a constant steady
state \((H^*,W^*,V^*)\), namely a solution of
\eqref{eq:svquasi} that is independent of both time and space. Our objective is to drive the flow toward
this state by adjusting the downstream gate while keeping the upstream
gate fixed.  We restrict attention to solutions near this target state
and therefore linearize the Saint--Venant equations around this state.  Moreover, we consider the subcritical steady state $(H^*,W^*,0)$ with
\[
  H^*>0,
 \qquad 0<W^*<\sqrt{gH^*}.
\]
Here \(W^*>0\) is the constant velocity along the channel toward the
downstream gate, whereas \(V^*=0\) means that there is no flow across
the channel.

\subsection{Linearization, boundary variables, and optimal control}

Near this steady state, let \(h\), \(w\), and \(v\) denote the deviations
of \(H\), \(W\), and \(V\) from their steady-state values, respectively.
Thus, we have
\[
 H=H^*+h,
 \qquad W=W^*+w,
 \qquad V=V^*+v=v.
\]

Then \(f:=(h,w,v)^\top\) satisfies the linearized Saint--Venant equations
\begin{equation}\label{eq:sv-linearized}
 f_t+A_1f_{x_1}+A_2f_{x_2}=Qf
 \qquad\text{in }(0,T)\times\Omega,
\end{equation}
with
\[
 A_1=
 \begin{pmatrix}
  W^*&H^*&0\\
  g&W^*&0\\
  0&0&W^*
 \end{pmatrix},
 \qquad
 A_2=
 \begin{pmatrix}
  0&0&H^*\\
  0&0&0\\
  g&0&0
 \end{pmatrix},
 \qquad
 Q=
 \begin{pmatrix}
  0&0&0\\
  0&-k&\ell\\
  0&-\ell&-k
 \end{pmatrix}.
\]

Set \(a:=\sqrt{g/H^*}\). On the upper wall
\(x_2=1\), according to \cite[Section~4]{YangYong2024},
\(a h-v\) and \(a h+v\) are
incoming and outgoing variables, respectively.   We prescribe the boundary condition
\begin{equation}\label{eq:sv-upper-wall-bc}
 \bigl[a h-v\bigr](t,x_1,1)=\bigl[a h+v\bigr](t,x_1,1),
 \qquad (t,x_1)\in(0,T)\times(0,L),
\end{equation}
which  corresponds to the gain \(K=1\).

At \(x_2=0\), the incoming and outgoing variables are \(a h+v\) and
\(v-a h\), respectively.  The corresponding boundary condition has
gain \(K=-1\):
\begin{equation}\label{eq:sv-lower-wall-bc}
 \bigl[a h+v\bigr](t,x_1,0)=-\bigl[v-a h\bigr](t,x_1,0),
 \qquad (t,x_1)\in(0,T)\times(0,L).
\end{equation}
Both \eqref{eq:sv-upper-wall-bc} and \eqref{eq:sv-lower-wall-bc} are
equivalent to \(v=0\).  Since \(V=V^*+v=v\) is the velocity component
perpendicular to these walls, these conditions precisely represent the impermeable
solid-wall boundary conditions.

At the downstream gate \(x_1=L\), the incoming variable is \(a h-w\),
whereas \(a h+w\) and \(v\) are outgoing.  A sluice gate regulates the
water discharge through it, which depends on the velocity component
perpendicular to the gate but not on the tangential velocity \(V\)
\cite[Sections~2 and~3]{MoralesHernandezEtAl2013}.
Accordingly, we take \(K=0\), use the scalar control \(u\) to prescribe
the incoming variable, and observe \(\zeta_+=(\zeta_{+,1},\zeta_{+,2})^\top\), where
\(\zeta_{+,1}(t,x_2):=a h(t,L,x_2)+w(t,L,x_2)\) and
\(\zeta_{+,2}(t,x_2):=v(t,L,x_2)\). The downstream boundary condition is
\begin{equation}\label{eq:sv-downstream-bc}
 \bigl[a h-w\bigr](t,L,x_2)=u(t,x_2),
 \qquad (t,x_2)\in(0,T)\times(0,1).
\end{equation}

At the upstream gate \(x_1=0\), the incoming variables are \(a h+w\)
and \(v\), whereas the outgoing variable is \(a h-w\). We keep this gate
fixed by imposing homogeneous conditions on both incoming variables:
\begin{equation}\label{eq:sv-upstream-bc}
 \bigl[a h+w\bigr](t,0,x_2)=v(t,0,x_2)=0,
 \qquad (t,x_2)\in(0,T)\times(0,1).
\end{equation}

Given \(T>0\), we define the approximate mechanical energy accumulated at the
downstream gate over \([0,T]\) as the sum of the kinetic and potential energies
of the water flow:
\begin{equation*}
 J_T(u):=\int_0^T\int_0^1\left[\frac{H^*}{2}(W^2+V^2)+\frac{g}{2}H^2\right](t,L,x_2)\dd x_2\dd t.
\end{equation*}
The kinetic-energy term penalizes high flow velocities at the downstream
gate. The potential-energy term penalizes large water depths, since high
water levels reduce the margin to overtopping. Using \(H=H^*+h\), \(W=W^*+w\), and \(V=v\),
\eqref{eq:sv-downstream-bc} and the definition of \(\zeta_+\) give
\begin{equation*}
\begin{aligned}
 J_T(u)=\int_0^T\int_0^1\bigg[&\frac{H^*}{2}
 \left(\left(W^*+\frac{\zeta_{+,1}-u}{2}\right)^2+\zeta_{+,2}^2\right)
 +\frac{g}{2}\left(H^*+\frac{u+\zeta_{+,1}}{2a}\right)^2\bigg]\dd x_2\dd t.
\end{aligned}
\end{equation*}

Motivated by the need for energy-efficient and safe operation of
controlled waterways \cite{VanOverloopEtAl2010}, we minimize \(J_T\)
over the downstream control \(u\).
This leads to the optimal control problem
\begin{equation}\label{eq:sv-optimal-control}
\begin{aligned}
   \underset{u\in L^2((0,T)\times(0,1))}{\operatorname{inf}}
 &J_T(u)\\
 \text{subject to}\quad
 &\begin{cases}
   \eqref{eq:sv-linearized},
     & \text{(equation)},\\
   \eqref{eq:sv-upper-wall-bc},\quad \eqref{eq:sv-lower-wall-bc},\quad
   \eqref{eq:sv-downstream-bc},\quad \eqref{eq:sv-upstream-bc},
     & \text{(boundary conditions)},\\
   f(0,\cdot)=f_0,
     & \text{(initial condition)}.
  \end{cases}
\end{aligned}
\end{equation}
Take
\[
 E(x)=(c-x_1)\operatorname{diag}\!\left(\frac{g}{H^*},1,1\right),
\]
with the positive constant $c$ suitably big.
Following \cite[Section~4]{YangYong2024}, one can verify that \(E\)
satisfies Assumption~\ref{ass:stabilizing-symmetrizer}. One can easily check that the requirements in Corollary \ref{cor:partial-boundary} for boundary conditions and cost functions also hold. Therefore,
Corollary~\ref{cor:partial-boundary} yields the exponential turnpike
property.

\subsection{Numerical illustration}

In this section, we perform a numerical experiment on the optimal boundary control problem \eqref{eq:sv-optimal-control} for the linearized
2-D Saint--Venant equations to illustrate the exponential turnpike property.  We discretize these equations by a
first-order upwind method. The control acts on the downstream incoming
variable, while both upstream incoming variables are set to zero. The
computation uses \(g=9.81\),
\(H^*=W^*=1\), \(k=0.15\), \(\ell=0.20\), \(L=10\), and \(T=40\).
The spatial and temporal step sizes are \(\Delta x_1=0.2\),
\(\Delta x_2=0.05\), and \(\Delta t=0.005\).
The initial data are
$
 (h_0,w_0,v_0)=\sin(\pi x_1/L)\bigl(0.05\cos(\pi x_2),\,0.03\cos(2\pi x_2),\,0.03\sin(\pi x_2)\bigr).
$

Figure~\ref{fig:sv-numerical-turnpike} displays, in order,
\[
\begin{aligned}
 &t\longmapsto\norm{f_T(t)-f_s}_{L^2(\Omega;\R^3)},\\
 &t\longmapsto\left(\int_{I_t}\int_0^1
 \abs{u_T(s,x_2)-u_s(x_2)}^2\dd x_2\dd s\right)^{1/2},\\
 &t\longmapsto\left(\int_{I_t}\int_0^1
 \norm{\zeta_{T,+}(s,x_2)-\zeta_{s,+}(x_2)}_{\R^2}^2\dd x_2\dd s\right)^{1/2}.
\end{aligned}
\]

\begin{center}
\begin{minipage}{\textwidth}
\centering
\captionsetup{hypcap=false}
\includegraphics[width=\textwidth]{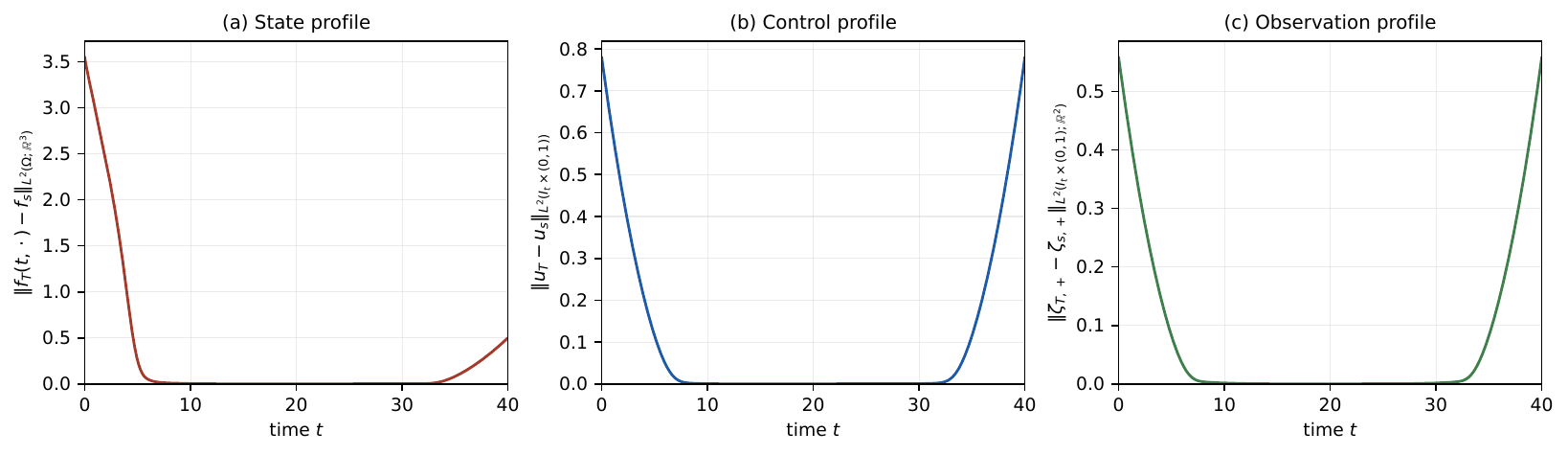}
\captionof{figure}{Numerical illustration of the turnpike phenomenon.
Panel (a) shows the pointwise-in-time state profile, while panels (b) and
(c) show the \(L^2\) profiles of the downstream control and the full
downstream observation over \(I_t\times(0,1)\).}
\label{fig:sv-numerical-turnpike}
\end{minipage}
\end{center}

All three profiles remain small over the central part of the horizon and
increase only near the temporal endpoints, consistently with the exponential
turnpike estimate.

\section{Conclusion}\label{sec:conclusion}

In this paper, we proved an exponential turnpike property for optimal
boundary control of multi-D first-order linear hyperbolic systems with
unbounded control and observation operators.  The result gives a
pointwise-in-time estimate for the state and \(L^2\)-estimates for the control
and observation, and also implies integral, measure, and localized turnpike
properties.  The proof combines dissipation estimates, input--output duality,
and strong convexity; we also establish stationary well-posedness and
illustrate the result numerically for the linearized 2-D Saint--Venant
equations.

In future work, we hope to extend the present analysis to quasilinear
first-order hyperbolic systems.  It would also be interesting to generalize
these results to a broader abstract operator framework, in the spirit of
\cite{NguyenTrelat2026,PorrettaZuazua2013,GruneSchallerSchiela2020}.

\appendix

\section{Well-posedness of the stationary boundary value problem}
\label{app:stationary-wellposedness}
In this appendix, we establish stationary well-posedness under
Assumptions~\ref{ass:stabilizing-symmetrizer} and
\ref{ass:boundary-dissipativity}. We first show that without
Assumption~\ref{ass:stabilizing-symmetrizer}, the stationary boundary
value problem need not be well-posed.

\subsection{A stationary Saint--Venant counterexample}
\label{app:static-counterexample}
Consider the source-free stationary Saint--Venant system on
\(\Omega=(0,1)^2\), linearized about the rest state
\((H^*,0,0)\), where \(H^*>0\):
\begin{equation}\label{eq:sv-static-counterexample-system}
\begin{pmatrix}
 0&H^*&0\\
 g&0&0\\
 0&0&0
\end{pmatrix}
\begin{pmatrix}h\\w\\v\end{pmatrix}_{x_1}
+
\begin{pmatrix}
 0&0&H^*\\
 0&0&0\\
 g&0&0
\end{pmatrix}
\begin{pmatrix}h\\w\\v\end{pmatrix}_{x_2}
=0.
\end{equation}
We impose the boundary conditions
\begin{equation}\label{eq:sv-static-counterexample-bc}
\begin{aligned}
 h(x_1,1)&=\sqrt{\frac{H^*}{g}}\,v(x_1,1),
 &\qquad
 h(x_1,0)&=-\sqrt{\frac{H^*}{g}}\,v(x_1,0),\\
 h(1,x_2)&=\sqrt{\frac{H^*}{g}}\,w(1,x_2),
 &
 h(0,x_2)&=-\sqrt{\frac{H^*}{g}}\,w(0,x_2).
\end{aligned}
\end{equation}
These conditions correspond to \(K=0\) and \(u=0\) in \eqref{eq:stationary-system}.
Moreover, one can easily verify that no matrix-valued function \(E=E(x)\) can satisfy
Assumption~\ref{ass:stabilizing-symmetrizer}. Note that for every pair of positive integers \(k,l\), the function
\begin{equation*}
 (h,w,v)(x_1,x_2)
 =
 \left(
 0,\,
l\sin(k\pi x_1)\cos(l\pi x_2),\,
 -k\cos(k\pi x_1)\sin(l\pi x_2)
 \right)
\end{equation*}
solves
\eqref{eq:sv-static-counterexample-system}--\eqref{eq:sv-static-counterexample-bc}. Hence the homogeneous stationary boundary value problem has infinitely many solutions and is not well-posed.

\subsection{Well-posedness}
Under Assumptions~\ref{ass:stabilizing-symmetrizer} and
\ref{ass:boundary-dissipativity}, we establish the well-posedness of
\eqref{eq:stationary-system}. The adjoint problem, for given
\(q\in\dbX\) and \(z\in\dbY\),
\begin{equation}\label{eq:stationary-adjoint-system}
\begin{cases}
 -\mathcal A^*g=q,&x\in\Omega,\\
 \Lambda_+\xi_++K^\top\Lambda_-\xi_-=z,&x\in\Gamma,
\end{cases}
\qquad \xi=\Pi^\top g,
\end{equation}
is treated similarly.

\begin{lemma}\label{lem:stationary-energy}
Under Assumptions~\ref{ass:stabilizing-symmetrizer} and
\ref{ass:boundary-dissipativity}, every stationary solution
\(f\in\dbX\) of \eqref{eq:stationary-system} with input
\(u\in\dbU\) satisfies
\begin{equation}\label{eq:stationary-coercive}
 \norm{f}_{\dbX}^2+\norm{\zeta_+}_{\dbY}^2
 \le C\norm{u}_{\dbU}^2.
\end{equation}
\end{lemma}

\begin{proof}
Since \(f\) is also a time-independent solution of
\eqref{eq:evolution-system} with initial datum \(f\) and constant input
\(u\), \eqref{eq:forward-uniform-estimate} gives, for every \(T>0\),
\begin{equation*}
 (1+T)\norm{f}_{\dbX}^2+T\norm{\zeta_+}_{\dbY}^2
 \le C\left(\norm{f}_{\dbX}^2+T\norm{u}_{\dbU}^2\right).
\end{equation*}
Since \(C\) is independent of \(T\), dividing by \(T\) and letting
\(T\to\infty\) proves \eqref{eq:stationary-coercive}.
\end{proof}

For the existence argument, we also need the smooth coercive estimate
contained in the proof of Lemma~\ref{lem:backward-energy-estimate}.
\begin{lemma}\label{lem:stationary-coercive-estimates}
Under Assumptions~\ref{ass:stabilizing-symmetrizer} and
\ref{ass:boundary-dissipativity}, there exists \(C>0\) such that
the following estimate holds. For every
\(g\in\mathcal D(\overline\Omega)\), set \(\xi:=\Pi^\top g\)
and \(z:=\Lambda_+\xi_++K^\top\Lambda_-\xi_-\). Then
\begin{equation}\label{eq:adjoint-static-coercive}
 \norm{g}_{\dbX}^2+\norm{\xi_-}_{\dbU}^2
 \leq C\left(
 \norm{\mathcal A^*g}_{\dbX}^2+\norm{z}_{\dbY}^2\right).
\end{equation}
\end{lemma}

\begin{proof}
Integration by parts gives
\begin{equation*}
 2\ip{\mathcal A^*g}{E^{-1}g}_{\dbX}
 =\int_\Omega g^\top\left[
 QE^{-1}+E^{-1}Q^\top
 -\sum_{j=1}^d\partial_{x_j}(E^{-1}A_j^\top)
 \right]g\dd x
 +\int_\Gamma g^\top E^{-1}A_\nu^\top g\dd\sigma.
\end{equation*}
By \eqref{eq:inverse-internal-dissipation}, \eqref{eq:dualdissipation},
and Young's inequality,
\begin{equation*}
 c\norm{g}_{\dbX}^2+c\norm{\xi_-}_{\dbU}^2
 \le -2\ip{\mathcal A^*g}{E^{-1}g}_{\dbX}+C\norm{z}_{\dbY}^2
 \le \frac c2\norm{g}_{\dbX}^2
 +C\left(\norm{\mathcal A^*g}_{\dbX}^2+\norm{z}_{\dbY}^2\right).
\end{equation*}
Absorbing the first term on the right proves the estimate.
\end{proof}

We now construct a stationary solution and use
\eqref{eq:stationary-coercive} to obtain uniqueness.
\begin{theorem}[Stationary well-posedness]\label{thm:static-wp}
Under Assumptions~\ref{ass:stabilizing-symmetrizer} and
\ref{ass:boundary-dissipativity}, there exists \(C>0\) such that,
for every \(u\in\dbU\), problem \eqref{eq:stationary-system}
has a unique solution \(f\in\dbX\), with \(\zeta_+\in\dbY\), and
\begin{equation*}
 \norm{f}_{\dbX}^2+\norm{\zeta_+}_{\dbY}^2
 \leq C\norm{u}_{\dbU}^2.
\end{equation*}
\end{theorem}

\begin{proof}
For the forward problem, consider the subspace
\begin{equation*}
 \mathscr R:=\left\{(\mathcal A^*g,-z):
 g\in\mathcal D(\overline\Omega)\right\}
 \subset\dbX\times\dbY,
\end{equation*}
where \(\xi:=\Pi^\top g\) and
\(z:=\Lambda_+\xi_++K^\top\Lambda_-\xi_-\), and define
\begin{equation*}
 \ell_u(\mathcal A^*g,-z):=\ip{u}{\Lambda_-\xi_-}_{\dbU}.
\end{equation*}
This functional is well defined. Indeed, if two smooth functions
\(g_1\) and \(g_2\) give the same element of \(\mathscr R\), then
\(\mathcal A^*(g_1-g_2)=0\) and the corresponding boundary quantity
\(z_1-z_2\) vanishes. By \eqref{eq:adjoint-static-coercive},
\(\xi_{1,-}-\xi_{2,-}=0\). Thus the value of
\(\ell_u\) is independent of the chosen representative. Moreover, the
boundedness of \(\Lambda_-\) and \eqref{eq:adjoint-static-coercive} give
\begin{equation*}
 \abs{\ell_u(\mathcal A^*g,-z)}
 \leq C\norm{u}_{\dbU}
 \left(\norm{\mathcal A^*g}_{\dbX}^2+\norm{z}_{\dbY}^2\right)^{1/2}.
\end{equation*}
Hence \(\ell_u\) is bounded on \(\mathscr R\). By the Hahn--Banach
theorem, it extends to \(\dbX\times\dbY\) without increasing its norm.
The Riesz representation theorem gives a pair
\((f,\zeta_+)\in\dbX\times\dbY\) such that
\begin{equation*}
 \ip{f}{p}_{\dbX}+\ip{\zeta_+}{r}_{\dbY}
 =\widetilde\ell_u(p,r),
 \qquad (p,r)\in\dbX\times\dbY.
\end{equation*}
Taking \((p,r)=(\mathcal A^*g,-z)\), we readily see that \(f\) solves
\eqref{eq:stationary-system}.
The estimate and uniqueness follow from Lemma~\ref{lem:stationary-energy}.
\end{proof}

\section{The boundary control and observation operators}
\label{app:boundary-operators}
In this appendix, we give the precise definitions of the operators in
\eqref{eq:abstract-control-system}. Equip \(D(\mathcal A)\) and
\(D(\mathcal A^*)\) with their graph norms, and let
\(D(\mathcal A^*)'\) denote the dual of \(D(\mathcal A^*)\), with
\(\dbX\) as the pivot space. For \(g\in D(\mathcal A^*)\), write \(\xi:=\Pi^\top g\). Motivated by the boundary term in the integration-by-parts identity, we define
\(\mathcal B\in\mathcal L\bigl(\dbU,D(\mathcal A^*)'\bigr)\) by
\begin{equation*}
 \left\langle\mathcal Bu,g\right\rangle_{D(\mathcal A^*)',D(\mathcal A^*)}
 :=\ip{u}{-\Lambda_-\xi_-}_{\dbU}.
\end{equation*}
As a result, we also have
\begin{equation*}
 \mathcal B^*\in\mathcal L\bigl(D(\mathcal A^*),\dbU\bigr),
 \qquad \mathcal B^*g:=-\Lambda_-\xi_-.
\end{equation*}
The quantity \(\mathcal Bu\) represents a boundary-supported functional
and, in general, is not an element of \(\dbX\). Hence
\(\mathcal B\notin\mathcal L(\dbU,\dbX)\).

For \(f\in D(\mathcal A)\), write \(\zeta:=\Pi^{-1}f\) on \(\Gamma\) and define
\(\mathcal Cf:=\zeta_+\). For a general \(f\in\dbX\), however, its
boundary trace is not defined; hence
\(\mathcal C\notin\mathcal L(\dbX,\dbY)\).

\bibliographystyle{abbrv}
\bibliography{references}

\begin{thebibliography}{10}

\bibitem{BastinCoron2016}
G.~Bastin and J.-M. Coron.
\newblock {\em Stability and Boundary Stabilization of {1-D} Hyperbolic
  Systems}, volume~88 of {\em Progress in Nonlinear Differential Equations and
  Their Applications}.
\newblock Birkh{\"a}user, Cham, 2016.

\bibitem{BenzoniGavageSerre2007}
S.~Benzoni-Gavage and D.~Serre.
\newblock {\em Multi-Dimensional Hyperbolic Partial Differential Equations:
  First-Order Systems and Applications}.
\newblock Oxford Mathematical Monographs. Oxford University Press, Oxford,
  2007.

\bibitem{BongartiHintermueller2024}
M.~Bongarti and M.~Hinterm{\"u}ller.
\newblock Optimal boundary control of the isothermal semilinear {Euler}
  equation for gas dynamics on a network.
\newblock {\em Applied Mathematics \& Optimization}, 89(2):36, 2024.

\bibitem{Brezis2011}
H.~Brezis.
\newblock {\em Functional Analysis, Sobolev Spaces and Partial Differential
  Equations}.
\newblock Universitext. Springer, New York, 2011.

\bibitem{ClementCoronShang2013}
F.~Cl{\'e}ment, J.-M. Coron, and P.~Shang.
\newblock Optimal control of cell mass and maturity in a model of follicular
  ovulation.
\newblock {\em SIAM Journal on Control and Optimization}, 51(2):824--847, 2013.

\bibitem{DorfmanSamuelsonSolow1958}
R.~Dorfman, P.~A. Samuelson, and R.~M. Solow.
\newblock {\em Linear Programming and Economic Analysis}.
\newblock McGraw-Hill, New York, 1958.

\bibitem{EsteveYagueGeshkovskiPighinZuazua2022}
C.~Esteve-Yag{\"u}e, B.~Geshkovski, D.~Pighin, and E.~Zuazua.
\newblock Turnpike in {Lipschitz}--nonlinear optimal control.
\newblock {\em Nonlinearity}, 35(4):1652--1701, 2022.

\bibitem{FaulwasserKordaJonesBonvin2017}
T.~Faulwasser, M.~Korda, C.~N. Jones, and D.~Bonvin.
\newblock On turnpike and dissipativity properties of continuous-time optimal
  control problems.
\newblock {\em Automatica}, 81:297--304, 2017.

\bibitem{GeshkovskiZuazua2022}
B.~Geshkovski and E.~Zuazua.
\newblock Turnpike in optimal control of {PDE}s, {ResNets}, and beyond.
\newblock {\em Acta Numerica}, 31:135--263, 2022.

\bibitem{GruneMueller2016}
L.~Gr{\"u}ne and M.~A. M{\"u}ller.
\newblock On the relation between strict dissipativity and turnpike properties.
\newblock {\em Systems \& Control Letters}, 90:45--53, 2016.

\bibitem{GruneSchallerSchiela2020}
L.~Gr{\"u}ne, M.~Schaller, and A.~Schiela.
\newblock Exponential sensitivity and turnpike analysis for linear quadratic
  optimal control of general evolution equations.
\newblock {\em Journal of Differential Equations}, 268(12):7311--7341, 2020.

\bibitem{Gugat2019ConvexHyperbolic}
M.~Gugat.
\newblock A turnpike result for convex hyperbolic optimal boundary control
  problems.
\newblock {\em Pure and Applied Functional Analysis}, 4(4):849--866, 2019.

\bibitem{GugatHante2019}
M.~Gugat and F.~M. Hante.
\newblock On the turnpike phenomenon for optimal boundary control problems with
  hyperbolic systems.
\newblock {\em SIAM Journal on Control and Optimization}, 57(1):264--289, 2019.

\bibitem{GugatHerty2023}
M.~Gugat and M.~Herty.
\newblock Turnpike properties of optimal boundary control problems with random
  linear hyperbolic systems.
\newblock {\em ESAIM: Control, Optimisation and Calculus of Variations},
  29:Paper No. 55, 27 pp., 2023.

\bibitem{GugatTrelatZuazua2016}
M.~Gugat, E.~Tr{\'e}lat, and E.~Zuazua.
\newblock Optimal {Neumann} control for the {1D} wave equation: Finite horizon,
  infinite horizon, boundary tracking terms and the turnpike property.
\newblock {\em Systems \& Control Letters}, 90:61--70, 2016.

\bibitem{HertyThein2024}
M.~Herty and F.~Thein.
\newblock Boundary feedback control for hyperbolic systems.
\newblock {\em ESAIM: Control, Optimisation and Calculus of Variations}, 30:71,
  2024.

\bibitem{HertyZhou2025}
M.~Herty and Y.~Zhou.
\newblock Exponential turnpike property for particle systems and mean-field
  limit.
\newblock {\em European Journal of Applied Mathematics}, 36(5):993--1011, 2025.

\bibitem{HuangTemam2014}
A.~Huang and R.~Temam.
\newblock The linear hyperbolic initial and boundary value problems in a domain
  with corners.
\newblock {\em Discrete and Continuous Dynamical Systems - B},
  19(6):1627--1665, 2014.

\bibitem{MoralesHernandezEtAl2013}
M.~Morales-Hern{\'a}ndez, J.~Murillo, and P.~Garc{\'i}a-Navarro.
\newblock The formulation of internal boundary conditions in unsteady {2-D}
  shallow water flows: Application to flood regulation.
\newblock {\em Water Resources Research}, 49(1):471--487, 2013.

\bibitem{NguyenTrelat2026}
H.-M. Nguyen and E.~Tr{\'e}lat.
\newblock Turnpike property of linear quadratic control problems with unbounded
  control operators.
\newblock {\em Advances in Differential Equations}, 31(3--4):219--252, 2026.

\bibitem{PorrettaZuazua2013}
A.~Porretta and E.~Zuazua.
\newblock Long time versus steady state optimal control.
\newblock {\em SIAM Journal on Control and Optimization}, 51(6):4242--4273,
  2013.

\bibitem{PorrettaZuazua2016Remarks}
A.~Porretta and E.~Zuazua.
\newblock Remarks on long time versus steady state optimal control.
\newblock In F.~Ancona, P.~Cannarsa, C.~Jones, and A.~Portaluri, editors, {\em
  Mathematical Paradigms of Climate Science}, volume~15 of {\em Springer INdAM
  Series}, pages 67--89. Springer, Cham, 2016.

\bibitem{SchmittUlbrich2021}
J.~M. Schmitt and S.~Ulbrich.
\newblock Optimal boundary control of hyperbolic balance laws with state
  constraints.
\newblock {\em SIAM Journal on Control and Optimization}, 59(2):1341--1369,
  2021.

\bibitem{SunYong2024}
J.~Sun and J.~Yong.
\newblock Turnpike properties for stochastic linear-quadratic optimal control
  problems with periodic coefficients.
\newblock {\em Journal of Differential Equations}, 400:189--229, 2024.

\bibitem{TrelatZhang2018}
E.~Tr{\'e}lat and C.~Zhang.
\newblock Integral and measure-turnpike properties for infinite-dimensional
  optimal control systems.
\newblock {\em Mathematics of Control, Signals, and Systems}, 30(1):Paper No.
  3, 34 pp., 2018.

\bibitem{TrelatZhangZuazua2018}
E.~Tr{\'e}lat, C.~Zhang, and E.~Zuazua.
\newblock Steady-state and periodic exponential turnpike property for optimal
  control problems in {Hilbert} spaces.
\newblock {\em SIAM Journal on Control and Optimization}, 56(2):1222--1252,
  2018.

\bibitem{TrelatZuazua2015}
E.~Tr{\'e}lat and E.~Zuazua.
\newblock The turnpike property in finite-dimensional nonlinear optimal
  control.
\newblock {\em Journal of Differential Equations}, 258(1):81--114, 2015.

\bibitem{TrelatZuazuaSurvey}
E.~Tr{\'e}lat and E.~Zuazua.
\newblock Turnpike in optimal control and beyond: A survey, 2025.
\newblock arXiv:2503.20342.

\bibitem{TucsnakWeiss2009}
M.~Tucsnak and G.~Weiss.
\newblock {\em Observation and Control for Operator Semigroups}.
\newblock Birkh{\"a}user Advanced Texts: Basler Lehrb{\"u}cher. Birkh{\"a}user,
  Basel, 2009.

\bibitem{VanOverloopEtAl2010}
P.~J. van Overloop, R.~R. Negenborn, S.~V. Weijs, W.~Malda, M.~R. Bruggers, and
  B.~De~Schutter.
\newblock Linking water and energy objectives in lowland areas through the
  application of model predictive control.
\newblock In {\em Proceedings of the 2010 IEEE International Conference on
  Control Applications}, pages 1887--1891, 2010.

\bibitem{YangYong2024}
H.~Yang and W.-A. Yong.
\newblock Feedback boundary control of multi-dimensional hyperbolic systems
  with relaxation.
\newblock {\em Automatica}, 167:111791, 2024.

\bibitem{YangYong2025}
H.~Yang and W.-A. Yong.
\newblock Boundary control of multidimensional discrete-velocity kinetic
  models.
\newblock {\em IEEE Transactions on Automatic Control}, 70(9):6183--6190, 2025.

\end{thebibliography}

\end{document}